\documentclass[11pt]{amsart}
\usepackage[mathscr]{euscript}
\DeclareFontFamily{OT1}{rsfs}{}
\DeclareFontShape{OT1}{rsfs}{n}{it}{<-> rsfs10}{}
\DeclareMathAlphabet{\curly}{OT1}{rsfs}{n}{it}

\makeatletter
\newcommand{\eqnum}{\refstepcounter{equation}\textup{\tagform@{\theequation}}}
\makeatother

\DeclareRobustCommand{\SkipTocEntry}[3]{}
\makeatletter
\newcommand\@dotsep{4.5}
\def\@tocline#1#2#3#4#5#6#7{\relax
  \ifnum #1>\c@tocdepth 
  \else
    \par \addpenalty\@secpenalty\addvspace{#2}%
    \begingroup \hyphenpenalty\@M
    \@ifempty{#4}{%
      \@tempdima\csname r@tocindent\number#1\endcsname\relax
    }{%
      \@tempdima#4\relax
    }%
    \parindent\z@ \leftskip#3\relax \advance\leftskip\@tempdima\relax
    \rightskip\@pnumwidth plus1em \parfillskip-\@pnumwidth
    #5\leavevmode #6\relax
    \leaders\hbox{$\m@th
      \mkern \@dotsep mu\hbox{.}\mkern \@dotsep mu$}\hfill
    \hbox to\@pnumwidth{\@tocpagenum{#7}}\par
    \nobreak
    \endgroup
  \fi}
\makeatother

\newcommand\beq[1]{\begin{equation}\label{#1}}
\newcommand\eeq{\end{equation}}
\newcommand\beqa{\begin{eqnarray*}}
\newcommand\eeqa{\end{eqnarray*}}

\title[Moduli stacks of semistable sheaves]
{Moduli stacks of semistable sheaves and 
representations of Ext-quivers}
\date{}
\author{Yukinobu Toda}

\usepackage{amscd}
\usepackage{amsmath}
\usepackage{amssymb}
\usepackage{amsthm}
\usepackage{float}
\usepackage[dvips]{graphicx}
\usepackage{xypic}

\usepackage{array}
\usepackage{amscd}

\DeclareFontFamily{U}{rsfs}{%
\skewchar\font127}
\DeclareFontShape{U}{rsfs}{m}{n}{%
<-6>rsfs5<6-8.5>rsfs7<8.5->rsfs10}{}
\DeclareSymbolFont{rsfs}{U}{rsfs}{m}{n}
\DeclareSymbolFontAlphabet
{\mathrsfs}{rsfs}
\DeclareRobustCommand*\rsfs{%
\@fontswitch\relax\mathrsfs}

\theoremstyle{plain}
\newtheorem{thm}{Theorem}[section]
\newtheorem{prop}[thm]{Proposition}
\newtheorem{lem}[thm]{Lemma}

\newtheorem{defi}[thm]{Definition}
\newtheorem{rmk}[thm]{Remark}
\newtheorem{cor}[thm]{Corollary}

\newtheorem{prop-defi}[thm]{Proposition-Definition}
\newtheorem{thm-defi}[thm]{Theorem-Definition}
\newtheorem{lem-defi}[thm]{Lemma-Definition}

\newcommand{\aA}{\mathcal{A}}

\newcommand{\cC}{\mathcal{C}}

\newcommand{\eE}{\mathcal{E}}
\newcommand{\fF}{\mathcal{F}}
\newcommand{\gG}{\mathcal{G}}
\newcommand{\hH}{\mathcal{H}}

\newcommand{\lL}{\mathcal{L}}
\newcommand{\mM}{\mathcal{M}}

\newcommand{\oO}{\mathcal{O}}

\newcommand{\sS}{\mathcal{S}}
\newcommand{\tT}{\mathcal{T}}
\newcommand{\uU}{\mathcal{U}}
\newcommand{\vV}{\mathcal{V}}
\newcommand{\wW}{\mathcal{W}}

\newcommand{\zZ}{\mathcal{Z}}

\newcommand{\Supp}{\mathop{\rm Supp}\nolimits}
\newcommand{\Hom}{\mathop{\rm Hom}\nolimits}

\newcommand{\dotimes}{\stackrel{\textbf{L}}{\otimes}}
\newcommand{\dR}{\mathbf{R}}

\newcommand{\id}{\textrm{id}}

\newcommand{\ch}{\mathop{\rm ch}\nolimits}

\newcommand{\Ext}{\mathop{\rm Ext}\nolimits}
\newcommand{\Spec}{\mathop{\rm Spec}\nolimits}

\newcommand{\Coh}{\mathop{\rm Coh}\nolimits}

\newcommand{\cneq}{\mathrel{\raise.095ex\hbox{:}\mkern-4.2mu=}}
\newcommand{\eqcn}{\mathrel{=\mkern-4.5mu\raise.095ex\hbox{:}}}

\newcommand{\gr}{\mathop{\rm gr}\nolimits}

\newcommand{\Aut}{\mathop{\rm Aut}\nolimits}

\newcommand{\Stab}{\mathop{\rm Stab}\nolimits}

\newcommand{\modu}{\mathop{\rm mod}\nolimits}
\newcommand{\Modu}{\mathop{\rm Mod}\nolimits}

\newcommand{\Imm}{\mathop{\rm Im}\nolimits}

 \newcommand{\RHom}{\mathop{\dR\mathrm{Hom}}\nolimits}
\newcommand{\Ker}{\mathop{\rm Ker}\nolimits}
\newcommand{\Ree}{\mathop{\rm Re}\nolimits}

\newcommand{\GL}{\mathop{\rm GL}\nolimits}

\newcommand{\tr}{\mathop{\rm tr}\nolimits}

\newcommand{\cl}{\mathop{\rm cl}\nolimits}

\newcommand{\sslash}{/\!\!/}
\newcommand{\lkakko}{[\![}
\newcommand{\rkakko}{]\!]}

\makeatletter
 \renewcommand{\theequation}{%
   \thesection.\arabic{equation}}
  \@addtoreset{equation}{section}
\makeatother

\begin{document}

\begin{abstract}
We show that the moduli stacks of semistable sheaves on 
smooth projective varieties are analytic locally 
on their coarse moduli spaces described in terms of 
representations of the associated Ext-quivers with 
convergent relations. 
When the underlying variety is a Calabi-Yau 3-fold,
our result describes the above moduli stacks as 
critical locus analytic locally on the
coarse moduli spaces. The results in this paper 
will be applied to the wall-crossing formula of 
Gopakumar-Vafa invariants
defined by Maulik and the author. 
\end{abstract}

\maketitle

\section{Introduction}
\subsection{Motivation}
The purpose of this paper is to give descriptions of 
moduli stacks of semistable sheaves on smooth projective varieties
in terms of quivers with (formal but convergent) relations, analytic locally 
on their coarse moduli spaces. 
The relevant quiver is the Ext-quiver
associated to the simple collection of coherent sheaves, determined by 
a polystable sheaf corresponding to a point of the coarse moduli space. 
Probably the main results have been folklore
for experts of moduli of sheaves
(at least on formal 
neighborhoods at closed points of the
coarse moduli space), 
but we cannot find any reference
and our purpose is to give precise statements and 
details of the proofs.
The main results in this paper 
will be used in the companion paper~\cite{TodGV}
in the proof of wall-crossing formula of 
Gopakumar-Vafa invariants introduced by Maulik and the author~\cite{MT}.  

\subsection{Results}
Let $X$
be a smooth projective variety over $\mathbb{C}$ 
and $\omega$ an ample divisor on it. 
Let $\mM_{\omega}$ be the moduli stack of 
$\omega$-Gieseker semistable sheaves on $X$, 
and $M_{\omega}$ the coarse moduli 
space of $S$-equivalence classes of them. 
There is a natural morphism
\begin{align*}
p_M \colon \mM_{\omega} \to M_{\omega}
\end{align*}
sending a semistable sheaf to its 
$S$-equivalence class. 
A closed point of $M_{\omega}$ corresponds to 
a polystable sheaf, i.e. 
a direct sum 
\begin{align}\label{intro:polyE}
E=\bigoplus_{i=1}^k V_i \otimes E_i
\end{align}
where each $E_1, \ldots, E_k$ are mutually 
non-isomorphic $\omega$-Gieseker stable sheaves with 
the same reduced Hilbert polynomials. 

The \textit{Ext-quiver} $Q$ 
associated to the collection $(E_1, \ldots, E_k)$
is defined by the quiver whose vertex is 
$\{1, \ldots, k\}$ and the number of arrows from 
$i$ to $j$ is the dimension of $\Ext^1(E_i, E_j)$. 
We denote by 
$\mM_{Q}$ the moduli stack of finite dimensional
$Q$-representations
with dimension vector $(\dim V_i)_{1\le i\le k}$, and 
$M_{Q}$ the coarse moduli space
of 
semi-simple $Q$-representations
with dimension vector as above. 
We have the natural morphism
\begin{align*}
p_Q \colon \mM_{Q} \to M_Q
\end{align*}
sending a $Q$-representation to its
semi-simplification. 
There is a point $0 \in M_Q$ represented by the  
semi-simple $Q$-representation
$\oplus_{i=1}^k V_i \otimes S_i$,
where $S_i$ is a
simple $Q$-representation corresponding to 
the vertex $i$.  
The following is the main result in this paper. 
\begin{thm}\label{intro:thm1}\emph{(Theorem~\ref{thm:precise})}
For $p \in M_{\omega}$
represented by a polystable sheaf (\ref{intro:polyE}), 
let $Q$ be the $\Ext$-quiver
associated to $(E_1, \ldots, E_k)$. 
Then there exist
analytic open neighborhoods
$p \in U \subset M_{\omega}$, 
$0 \in V \subset M_{Q}$, 
closed analytic substack 
$\zZ \subset p_Q^{-1}(V)$
with 
the natural morphism to its 
coarse moduli space 
$p_Q \colon \zZ \to Z$ 
and the commutative isomorphisms
\begin{align*}
\xymatrix{
\zZ \ar[d]_-{p_Q} \ar[r]^-{\cong} & p_M^{-1}(U) 
\ar[d]^{p_M}
\\
Z \ar[r]^-{\cong} & U.
}
\end{align*}
\end{thm}
Indeed we can define the (formal but convergent) relation $I$
of the Ext-quiver $Q$, using 
the minimal $A_{\infty}$-structure of the dg-category
generated by $(E_1, \ldots, E_k)$. 
The convergence of $I$ will be proved 
by generalizing 
the gauge theory arguments of~\cite{MR1950958, JuTu}
for deformations of vector bundles 
 to the case of resolutions of 
coherent sheaves by complexes of 
vector bundles. 
The substack $\zZ \subset p_Q^{-1}(V)$ is then defined to 
be the stack of $Q$-representations satisfying the relation $I$. 

When $X$ is a smooth projective Calabi-Yau (CY) 3-fold, 
we can take the relation $I$ to be 
the derivation of a convergent super-potential of the quiver $Q$. 
So we have the following corollary of 
Theorem~\ref{intro:thm1}: 
\begin{cor}\label{intro:thm2}\emph{(Corollary~\ref{cor:CY3})}
In the situation of Theorem~\ref{intro:thm1}, 
suppose that $X$ is a smooth projective CY 3-fold. 
Then 
there is a morphism of complex analytic stacks
$W \colon p_Q^{-1}(V) \to \mathbb{C}$
such that 
\begin{align*}
\zZ=\{dW=0\} \stackrel{\cong}{\to} p_M^{-1}(U). 
\end{align*}
\end{cor}
A result similar to (\ref{cor:CY3}) was already proved in~\cite{JS, BBBJ}, 
where the stack $\mM_{\omega}$ is  
described as a critical locus locally on $\mM_{\omega}$. 
Our description is more global, as
we describe the stack $\mM_{\omega}$
as a critical locus
on the preimage of an open subset 
of the coarse moduli space $M_{\omega}$. 
The result of Corollary~\ref{cor:CY3} is 
also compatible with 
the $d$-critical structure introduced by Joyce~\cite{JoyceD}. 
By~\cite{PTVV}, the stack $\mM_{\omega}$ 
is a truncation of a derived scheme with a 
$(-1)$-shifted symplectic structure~\cite{PTVV}.
Using this fact, it is proved in~\cite{BBBJ} that 
the stack $\mM_{\omega}$
has a canonical 
$d$-critical structure. 
From the construction of $W$ in Corollary~\ref{intro:thm2},
the data 
$(p_M^{-1}(U), p_Q^{-1}(V), W)$ 
is shown to give a $d$-critical 
chart of the $d$-critical stack $\mM_{\omega}$
 (see~\cite[Appendix~A]{TodGV}).  

In the case of moduli spaces of one dimensional sheaves, 
we also investigate the wall-crossing 
phenomena of these moduli spaces with respect to the twisted 
stability. Let 
$A(X)_{\mathbb{C}}$ be the complexified ample cone of $X$
and take an element
\begin{align*}
\sigma=B+i\omega
\in A(X)_{\mathbb{C}}.
\end{align*}
Let 
$M_{\sigma}$
be the coarse moduli space of one dimensional 
$B$-twisted $\omega$-semistable 
sheaves on $X$. 
We will see that the 
result of Theorem~\ref{intro:thm1} is 
also applied for the
moduli space $M_{\sigma}$ of twisted semistable sheaves. 
If we take $\sigma^{+} \in A(X)_{\mathbb{C}}$ to be
sufficiently close to $\sigma$, we have the 
natural projective morphism
\begin{align}\label{intro:mor:wall}
q_M \colon M_{\sigma^{+}} \to M_{\sigma}.
\end{align}

\begin{thm}\emph{(Theorem~\ref{thm:onedim})}\label{intro:thm:onedim}
For $p\in M_{\sigma}$, 
let an open subset $p \in U \subset M_{\sigma}$, 
a quiver $Q$, 
and an analytic space $Z$
be as in Theorem~\ref{intro:thm1}. 
Then there is a stability condition $\xi$ on the category of 
$Q$-representations such that 
we have the commutative diagram
of isomorphisms
\begin{align}\label{intro:dia:onedim}
\xymatrix{
Z_{\xi} \ar[r]^-{\cong} \ar[d] & q_M^{-1}(U) \ar[d]^-{q_M} \\
Z \ar[r]^-{\cong} & U.
}
\end{align}
Here $Z_{\xi}$ is the coarse moduli space of $\xi$-semistable 
$Q$-representations satisfying the relation $I$. 
\end{thm}
When $X$ is a K3 surface, 
the morphism (\ref{intro:mor:wall})
was studied by Arbarello-Sacc\`a~\cite{Sacca}. 
In this case, they showed that the morphism (\ref{intro:mor:wall}) 
is analytic locally on $M_{\sigma}$
described as a symplectic resolution of singularities
of Nakajima quiver varieties via variation of 
stability conditions of representations of quivers. 
One can check that the result of Theorem~\ref{intro:thm:onedim}
gives the same description of the morphism (\ref{intro:mor:wall})
as in~\cite{Sacca}, if we know the 
formality of the dg-algebra $\RHom (E, E)$ for 
a polystable sheaf $[E] \in M_{\sigma}$. 

The results of 
Corollary~\ref{intro:thm2} and
Theorem~\ref{intro:thm:onedim} 
will be used in~\cite{TodGV} to show the wall-crossing formula 
of (generalization of) Gopakumar-Vafa (GV) invariants introduced by 
Maulik and the author~\cite{MT}. 
The idea is roughly speaking as follows. 
In~\cite{TodGV}, we construct some perverse sheaves 
$\phi_{M_{\sigma^{+}}}$, 
$\phi_{M_{\sigma}}$ on the moduli 
spaces $M_{\sigma^{+}}$, $M_{\sigma}$
 in (\ref{intro:mor:wall})
respectively, following the analogy of 
BPS sheaves introduced by 
Davison-Meinhardt~\cite{DaMe}. 
It turns out that there 
is a natural morphism
\begin{align}\label{nat:isom}
\phi_{M_{\sigma}} \to \dR q_{M\ast}
\phi_{M_{\sigma^{+}}}
\end{align}
and we want to show that the above morphism 
is an isomorphism. 
The results of 
Corollary~\ref{intro:thm2} and 
Theorem~\ref{intro:thm:onedim}
enable us to reduce to the 
case of quivers with convergent super-potentials. 
In the case of quivers with super-potentials,
the similar question was addressed and solved in~\cite{DaMe}, 
and 
we can use the results and arguments in \textit{loc.cit.} 
 to show that (\ref{nat:isom}) is an isomorphism. 

In a similar way, using the result of Corollary~\ref{intro:thm2} 
it should be possible to reduce several problems in 
Donaldson-Thomas (DT) theory on CY 3-folds to the case of 
representations of quivers with convergent super-potentials, which 
is easier in many cases. 
For example it
is recently announced by Davison-Meinhardt 
that the integrality conjecture of generalized DT invariants~\cite{JS, K-S}
on CY 3-folds
can be proved using the result of Corollary~\ref{intro:thm2}.

\subsection{Plan of the paper}
The organization of this paper is as follows.
In Section~\ref{sec:quiver}, 
we introduce the notion of quivers with convergent relations
and construct the moduli spaces of their representations. 
In Section~\ref{sec:moduli}, we fix some notation 
on the moduli spaces of semistable sheaves
and state the precise form of Theorem~\ref{intro:thm1}. 
In Section~\ref{sec:deform}, we describe deformation theory 
of coherent sheaves in terms of minimal $A_{\infty}$-structures. 
In Section~\ref{sec:thm}, we complete the proof of Theorem~\ref{intro:thm1}. 
In Section~\ref{sec:NC}, we recall NC deformation theory and relate
it with the result of Theorem~\ref{intro:thm1}. 
In Section~\ref{sec:one}, we prove Theorem~\ref{intro:thm:onedim}. 

\subsection{Acknowledgements}
The author is grateful to Ben Davison,
Davesh Maulik
for many useful discussions, 
Bingyu Xia for a comment on analytic Hilbert quotients, 
and the referee for the careful check of the paper 
with several comments. 
The author is supported by World Premier International Research Center
Initiative (WPI initiative), MEXT, Japan, and Grant-in Aid for Scientific
Research grant (No. 26287002) from MEXT, Japan.

\section{Quivers with convergent relations}\label{sec:quiver}
In this section, we recall some basic notions on 
quivers, their representations and moduli spaces. 
We also introduce the concept of convergent relations of quivers, 
and moduli spaces of quiver representations satisfying 
such relations. 
\subsection{Representations of quivers}
Recall that a \textit{quiver} $Q$ consists data
\begin{align*}
Q=(V(Q), E(Q), s, t)
\end{align*}
where $V(Q), E(Q)$ are finite sets 
and $s, t$ are maps
\begin{align*}
s, t \colon E(Q) \to V(Q).
\end{align*}
The set $V(Q)$ is the set of vertices and
$E(Q)$ is the set of edges. 
For $e \in E(Q)$, 
$s(e)$ is the source of $e$
and $t(e)$ is the target of $e$. 
For $i, j \in V(Q)$, we use the following notation
\begin{align}\label{Eab}
E_{i, j} \cneq \{e \in E(Q) : 
s(e)=i, t(e)=j\}
\end{align}
i.e. $E_{i, j}$ is the set of edges 
from $i$ to $j$. 
 
A \textit{$Q$-representation} consists of
data
\begin{align}\label{rep:Q}
\mathbb{V}=\{
(V_i, u_e) : \ i \in V(Q),  \ e \in E(Q), \ 
u_e \colon V_{s(e)} \to V_{t(e)}\}
\end{align}
where $V_i$ is a finite dimensional 
$\mathbb{C}$-vector space 
and $u_e$ is a linear map. 
For a $Q$-representation (\ref{rep:Q}), the vector
\begin{align}\label{m:vect}
\vec{m}=(m_i)_{i \in V(Q)}, \ 
m_i=\dim V_i
\end{align}
is called the \textit{dimension vector}. 

Given a dimension vector (\ref{m:vect}), 
let $V_i$ be a $\mathbb{C}$-vector space with 
dimension $m_i$. 
Let us set 
\begin{align*}
G \cneq \prod_{i \in Q(V)} \GL(V_i), \ 
\mathrm{Rep}_Q(\vec{m}) \cneq \prod_{e \in E(V)} \Hom(V_{s(e)}, V_{t(e)}).
\end{align*}
The algebraic group $G$ acts on $\mathrm{Rep}_Q(\vec{m})$ by 
\begin{align}\label{G:act}
g \cdot u=\{g_{t(e)}^{-1} \circ u_e \circ g_{s(e)}\}_{e\in E(Q)}
\end{align}
for $g=(g_i)_{i \in V(Q)} \in G$ 
and $u=(u_e)_{e\in E(Q)}$. 
A $Q$-representation with dimension vector $\vec{m}$ is 
determined by a point in $\mathrm{Rep}_Q(\vec{m})$
up to $G$-action. 
The moduli stack of $Q$-representations with 
dimension vector $\vec{m}$ is given by the 
quotient stack 
\begin{align*}
\mM_{Q}(\vec{m}) \cneq \left[ \mathrm{Rep}_Q(\vec{m})/G \right]. 
\end{align*}
It has the coarse moduli space, given by
\begin{align}\label{mor:coarse}
p_Q \colon
\mM_{Q}(\vec{m}) \to M_{Q}(\vec{m}) \cneq \mathrm{Rep}_Q(\vec{m}) \sslash G. 
\end{align}
Here in general, if a reductive algebraic group $G$ acts on 
an affine scheme $Y=\Spec R$, then 
its affine GIT quotient is given by
\begin{align*}
Y\sslash G \cneq \Spec R^G. 
\end{align*}
For  
two points in $x_1, x_2 \in Y$, they are mapped to the same 
point in $Y \sslash G$ iff their $G$-orbit closures intersect, 
i.e.
\begin{align*}
\overline{G \cdot x_1} \cap \overline{G \cdot x_2} \neq \emptyset.
\end{align*}
In the case of $G$-action on 
$\mathrm{Rep}_{Q}(\vec{m})$, the above condition 
is also equivalent to that the 
corresponding $Q$-representations
have the isomorphic semi-simplifications. 
The quotient space $M_Q(\vec{m})$ parametrizes 
semi-simple $Q$-representations with dimension 
vector $\vec{m}$, and the map (\ref{mor:coarse})
sends a $Q$-representation to its 
semi-simplification (see~\cite[Section~5]{MR2004218}, \cite[Section~3]{MR1315461} for details). 

For $i \in V(Q)$, let
 $S_i$ be the simple $Q$-representation 
corresponding to the vertex $i$, i.e. 
it is the unique $Q$-representation 
with dimension vector 
$m_i=1$ and $m_j =0$ for $j\neq i$. 
The point $0 \in \mathrm{Rep}_Q(\vec{m})$ and its 
image $0 \in M_Q(\vec{m})$ by the map (\ref{mor:coarse})
correspond to semi-simple $Q$-representation 
$\oplus_{i\in V(Q)}V_i \otimes S_i$. 
A $Q$-representation (\ref{rep:Q})
is called \textit{nilpotent} if any sufficiently large number of 
compositions of the linear maps $u_e$ becomes zero. 
It is easy to see that 
a $Q$-representation is nilpotent iff it is 
an iterated extensions of simple objects 
$\{S_i\}_{i \in V(Q)}$. 
In particular, 
the fiber 
\begin{align*}
p_Q^{-1}(0) \subset \mM_Q(\vec{m})
\end{align*}
 for the morphism (\ref{mor:coarse})
consists of nilpotent $Q$-representations
with dimension vector $\vec{m}$. 
\subsection{Quivers with convergent relations}\label{subsec:conv}
Recall that a \textit{path}
of a quiver $Q$ 
is a composition of edges in $Q$
\begin{align*}
e_1 e_2 \ldots e_n, \ e_i \in E(Q), \ t(e_i)=s(e_{i+1}). 
\end{align*}
The number $n$ above is called the \textit{length} of the path. 
The \textit{path algebra} of 
a quiver $Q$ is
a $\mathbb{C}$-vector space spanned by 
paths in $Q$:
\begin{align*}
\mathbb{C}[Q] \cneq
\bigoplus_{n\ge 0}
\bigoplus_{e_1, \ldots, e_n \in E(Q), t(e_i)=s(e_{i+1})} \mathbb{C} \cdot e_1 e_2 \ldots e_n.
\end{align*}
Here a path of length zero is a trivial path 
at each vertex of $Q$, and 
the product on $\mathbb{C}[Q]$ is defined by 
composition of paths. 
By taking the completion of $\mathbb{C}[Q]$ with respect to the 
length of the path, 
we obtain the formal 
path algebra:
\begin{align*}
\mathbb{C}\lkakko Q \rkakko \cneq 
\prod_{n\ge 0}
\bigoplus_{e_1, \ldots, e_n \in E(Q), t(e_i)=s(e_{i+1})} \mathbb{C} \cdot e_1 e_2 \ldots e_n.
\end{align*}
Note that an element $f \in \mathbb{C}\lkakko Q \rkakko$
is written as 
\begin{align}\label{f:element}
f=\sum_{n\ge 0, \{1, \ldots, n+1\} \stackrel{\psi}{\to} V(Q)}
\sum_{e_i \in E_{\psi(i), \psi(i+1)}}
a_{\psi, e_{\bullet}} \cdot e_1 e_2\ldots e_{n}. 
\end{align}
Here 
$a_{\psi, e_{\bullet}} \in \mathbb{C}$, 
$e_{\bullet}=(e_1, \ldots, e_n)$ and 
$E_{\psi(i), \psi(i+1)}$ is defined as in (\ref{Eab}). 
The above element $f$ lies in $\mathbb{C}[Q]$ iff
$a_{\psi, e_{\bullet}}=0$ for $n\gg 0$. 
\begin{defi}
We define the subalgebra
\begin{align*}
\mathbb{C}\{ Q\} \subset \mathbb{C}\lkakko Q \rkakko
\end{align*}
to be elements (\ref{f:element}) 
such that $\lvert a_{\psi, e_{\bullet}} \rvert <C^n$ for 
some constant $C>0$ which is independent of $n$. 
\end{defi}
Note that $\mathbb{C}\{Q\}$ contains $\mathbb{C}[Q]$ as 
a subalgebra. 
For an element $f \in \mathbb{C}\{Q\}$, 
we write it as (\ref{f:element})
and consider 
the following
$\Hom(V_a, V_b)$-valued formal function 
of $u=(u_e)_{e \in E(Q)} \in \mathrm{Rep}_Q(\vec{m})$
\begin{align}\label{Vab}
&f(a, b, \vec{m})\cneq \\
&\notag \sum_{\begin{subarray}{c}
n\ge 0, 
\psi \colon 
 \{1, \ldots, n+1\} \to V(Q), \\
\psi(1)=a, \psi(n+1)=b
\end{subarray}}
\sum_{e_i \in E_{\psi(i), \psi(i+1)}}
a_{\psi, e_{\bullet}} \cdot u_{e_n} \circ \cdots \circ u_{e_2} \circ u_{e_1}. 
\end{align}
By the definition of $\mathbb{C}\{Q\}$, 
the above $\Hom(V_a, V_b)$-valued formal function on $\mathrm{Rep}_Q(\vec{m})$ 
has a 
convergent radius. 
So 
there is an analytic open neighborhood 
\begin{align}\label{open:V}
0 \in \uU \subset \mathrm{Rep}_Q(\vec{m})
\end{align}
such that 
the function (\ref{Vab})
absolutely converges on it and 
determines the complex analytic map
\begin{align*}
f(a, b, \vec{m}) \colon \uU \to \Hom(V_a, V_b). 
\end{align*}
In particular,  
the equations
$f(a, b, \vec{m})=0$ for all $a, b \in V(Q)$
determines the 
closed complex analytic subspace of $\uU$. 

\subsection{Saturated open subsets}
We will extend the arguments in the previous 
subsection to a preimage
of an open subset in $\mathrm{Rep}_Q(\vec{m}) \sslash G$. 
Before doing this, we prepare
some general definitions and lemmas for 
the action of a reductive algebraic group on 
affine schemes or analytic spaces.   
\begin{defi}\label{def:saturated}
Let $G$ be a reductive group acting on
an affine algebraic $\mathbb{C}$-scheme $Y$. 
Then an analytic open set $U \subset Y$ is called
saturated if for any $x \in U$, 
the orbit closure $\overline{G \cdot x} \subset Y$ 
is contained in $U$. 
\end{defi}
Note that a saturated open subset is in particular 
$G$-invariant. 
Let 
\begin{align}\label{quot:Y}
\pi_Y \colon Y \to Y\sslash G
\end{align}
be the quotient map and 
$V \subset Y \sslash G$ be an analytic open 
subset. Then $\pi_Y^{-1}(V)$ is obviously saturated. 
Indeed, the converse is also true. 
In order to see this, we recall 
the following fact on the topology of affine GIT 
quotient 
$Y\sslash G$. 
\begin{thm}\emph{(\cite{MR819554, MR1040861})}\label{thm:KN}
In the situation of Definition~\ref{def:saturated}, 
let $K \subset G$ be a maximal compact subgroup of $G$. 
Then there is a 
$K$-invariant 
closed subset $S \subset Y$ in analytic topology, 
called \textit{Kempf-Ness set}, 
satisfying the following: 
for any $x \in S$ the $G$-orbit $G \cdot x$ is closed in $Y$ 
and 
the inclusion $S \subset Y$ induces 
the homeomorphism
\begin{align}\label{induced}
\iota \colon 
S/K \stackrel{\cong}{\to} Y\sslash G.
\end{align}
Here the topology of $S/K$ is a quotient topology induced from the 
analytic topology of $S$, and that of 
$Y\sslash G$
is the analytic topology. 
In particular, the analytic topology of $Y\sslash G$ is 
the quotient topology induced from the analytic topology of $Y$. 
\end{thm}

The following lemma follows from the above theorem: 
\begin{lem}\label{lem:saturated}
In the situation of Definition~\ref{def:saturated}, 
an analytic open subset $U \subset Y$ is saturated iff 
there is an analytic open set $V\subset Y\sslash G$
such that $U=\pi_Y^{-1}(V)$ 
where $\pi_Y \colon Y \to Y\sslash G$ is the quotient morphism. 
\end{lem}
\begin{proof}
For $x \in U$ and $y \in Y$, suppose that 
$\pi_Y(x)=\pi_Y(y)$, i.e. 
$\overline{G \cdot x}$ and $\overline{G \cdot y}$ intersect.
Since $U$ is saturated, we have 
$\overline{G \cdot x} \subset U$. 
Then we have 
$\overline{G \cdot y} \cap U \neq \emptyset$, and 
since $U$ is open there is $g \in G$ such that $g \cdot y \in U$. 
Therefore we have $y \in U$.
This implies that there is a subset $V \subset Y\sslash G$ such
that $U=\pi_Y^{-1}(V)$. 
By Theorem~\ref{thm:KN}, the subset $V$ is analytic open, hence 
the lemma holds. 
\end{proof}
We also have the following lemma. 
\begin{lem}\label{lem:saturated2}
In the situation of Definition~\ref{def:saturated}, 
let $y \in Y$ be a $G$-fixed point 
and $U \subset Y$ a $G$-invariant analytic open subset
with $y \in U$. 
Then there is an analytic open subset 
$U' \subset Y$, which is saturated and satisfies
$0 \in y \in U' \subset U$. 
\end{lem}
\begin{proof}
Let $S \subset Y$ be the Kempf-Ness set as in 
Theorem~\ref{thm:KN}. 
Since $y \in Y$ is $G$-fixed, 
we have $y\in S$ by the homeomorphism (\ref{induced}). 
Then we have 
$y \in S \cap U$, and $S \cap U$ is a $K$-invariant open subset in $S$. 
Therefore we have $S \cap U=\pi_S^{-1}(V)$
for some open subset $V \subset S/K$, 
where $\pi_S \colon S \to S/K$ is the quotient map. 
Since the map $\tau$ in (\ref{induced}) is a homeomorphism, 
the subset $\iota(V) \subset Y\sslash G$ is
open. 
We set a saturated open subset $U' \subset Y$ to be 
$U'=\pi_Y^{-1}(\iota(V))$
for the quotient map (\ref{quot:Y}). 
Since 
$\pi_S(y) \in V$, we have 
$y \in U'$. It is enough to
check that $U' \subset U$. 
By the construction of $U'$, 
for $x \in U'$
there is $z \in S \cap U$
such that $\pi_Y(x)=\pi_Y(z)$, i.e. 
the closures of $G \cdot x$ and $G \cdot z$ intersect. 
Since $G \cdot z$ is closed, 
we have 
$z \in \overline{G \cdot x}$. 
Therefore there is $g \in G$ such that 
$g \cdot x \in U$. 
Since $U$ is $G$-invariant, we have $x \in U$, 
hence the lemma is proved. 
\end{proof}

\subsection{Analytic Hilbert quotients}
Later we will take GIT-type quotients for non-algebraic complex 
analytic spaces. Here we recall the basic notions for such quotients. 
The following definition appears in~\cite{MR1631577, MR3394374}
for reduced complex analytic spaces. 
\begin{defi}\label{def:Hquot}
Let $G$ be a reductive algebraic group 
acting on a complex analytic space $Z$. 
Then a complex analytic space $Z \sslash G$
together with a morphism
\begin{align}\label{AHilb}
\pi_Z \colon Z \to Z \sslash G
\end{align}
is called an \textit{analytic Hilbert quotient} if the following conditions 
hold: 
\begin{enumerate}
\item $\pi_Z$ is a locally Stein map, i.e. 
there is an open cover $Z\sslash G=\cup_{\lambda} \uU_{\lambda}$ 
by Stein open subsets $\uU_{\lambda}$ such that 
$\pi_Z^{-1}(\uU_{\lambda})$ is Stein. 
\item We have $(\pi_{Z\ast}\oO_Z)^G=\oO_{Z\sslash G}$. 
\end{enumerate}
\end{defi}
An analytic Hilbert quotient
is known to exist when $Z$ is a reduced Stein space, 
which is unique up to isomorphism~\cite{MR1103041}.  
In~\cite{MR1631577, MR3394374}, analytic Hilbert
quotients are discussed under the assumption that $Z$ is reduced. 
It seems that such quotients for non-reduced analytic spaces are
not available in literatures. 
We don't develop generality of such quotients for 
non-reduced analytic spaces, but show the 
existence of such quotients in some special cases 
discussed below, and their universality.

We
show the following lemma on the existence
of analytic Hilbert quotients, which 
may be well-known, but we include it here 
as we cannot find a reference. 
\begin{lem}\label{lem:aquot}
Let $Y$ be an affine algebraic $\mathbb{C}$-scheme 
with $G$-action. 
Then for the affine GIT quotient 
$\pi_Y \colon Y \to Y \sslash G$, 
its analytification 
\begin{align*}\pi_Y^{an} \colon 
Y^{an} \to (Y\sslash G)^{an}
\end{align*}
is an analytic Hilbert quotient. 
\end{lem}
\begin{proof}
The condition (1) in Definition~\ref{def:Hquot} is obvious
as $Y^{an}$ and $(Y\sslash G)^{an}$ are Stein, so 
we only prove (2). 
First suppose that 
$Y=\mathbb{C}^n$ and the $G$-action on it is linear. 
In this case,
the condition (2) in Definition~\ref{def:Hquot}
is proved in~\cite{MR0423398}. 
In general, there is a $G$-invariant 
closed embedding $Y \subset \mathbb{C}^n$ 
where $G$ acts on $\mathbb{C}^n$ linearly, 
and the commutative diagram
\begin{align}\label{dia:Y}
\xymatrix{
Y 
\ar@<-0.3ex>@{^{(}->}[r]
\ar[d]_{\pi_{Y}} &
\mathbb{C}^n \ar[d]^{\pi_{\mathbb{C}^n}} \\
Y \sslash G \ar@<-0.3ex>@{^{(}->}[r] & \mathbb{C}^n \sslash G. 
}
\end{align}
Here since $G$ is reductive, 
the functor $(-)^G$ sending a
$G$-representation to its $G$-invariant part is exact. 
So the natural map 
$\Gamma(\oO_{\mathbb{C}^n})^G \to \Gamma(\oO_Y)^G$
is surjective, so the bottom 
arrow of (\ref{dia:Y}) is a closed embedding. 

By taking the analytification
of (\ref{dia:Y}), we obtain the commutative diagram of analytic sheaves
on $(\mathbb{C}^n \sslash G)^{an}$
\begin{align}\label{dia:sheaf}
\xymatrix{
\oO_{(\mathbb{C}^n \sslash G)^{an}} 
\ar[r]^(.4){\cong}
\ar[d] & (\pi^{an}_{\mathbb{C}^n \ast}\oO_{(\mathbb{C}^n)^{an}})^G \ar[d] \\
\oO_{(Y\sslash G)^{an}} \ar[r] & (\pi^{an}_{Y \ast}\oO_{Y^{an}})^G. 
}
\end{align}
Since $\pi^{an}_{\mathbb{C}^n}$ is locally Stein, and 
the functor $(-)^G$ is exact, 
the vertical arrows of (\ref{dia:sheaf})
are surjections. 
Therefore the bottom arrow of (\ref{dia:sheaf})
is surjective. 
Also 
as $\oO_{Y\sslash G}=(\pi_{Y\ast}\oO_Y)^{G}$
for Zariski sheaves, 
we have an injection 
$\oO_{Y\sslash G} \hookrightarrow \pi_{Y\ast}\oO_Y$, 
which is also injective after taking 
completions at each closed 
point of $\oO_{Y\sslash G}$.
Hence the bottom arrow of (\ref{dia:sheaf}) 
is also injective, so it is an isomorphism, i.e. 
$\pi_Y^{an}$ satisfies the condition (2) in Definition~\ref{def:Hquot}.
\end{proof}
By Lemma~\ref{lem:aquot}, for 
an analytic open subset $U \subset Y\sslash G$
the map 
\begin{align}\label{piY:U}
\pi_Y \colon 
\pi_Y^{-1}(U) \to U
\end{align}
is an analytic Hilbert quotient of $\pi_Y^{-1}(U)$. 
We also have the following lemma: 
\begin{lem}\label{lem:Zquot}
Let $Z \subset \pi_Y^{-1}(U)$ be a $G$-invariant 
closed analytic subspace. 
Then there is a 
closed analytic subspace $Z \sslash G \hookrightarrow U$
and an analytic Hilbert quotient $\pi_Z \colon Z \to Z \sslash G$. 
\end{lem}
\begin{proof}
Since (\ref{piY:U}) is an analytic Hilbert quotient 
and the functor $(-)^G$ is exact, we have the surjection
\begin{align*}
\oO_U=(\pi_{Y\ast}\oO_{\pi_Y^{-1}(U)})^G \twoheadrightarrow
(\pi_{Y\ast}\oO_Z)^G. 
\end{align*}
Therefore by setting $Z\sslash G$ to be 
the complex analytic subspace of $U$ defined by the ideal 
of the above kernel, we obtain the 
analytic Hilbert quotient $\pi_Z=\pi_Y|_{Z} \colon Z \to Z \sslash G$. 
\end{proof}
By gluing the above construction, we have the following 
lemma: 
\begin{lem}\label{lem:Zquot2}
Let $Y$ be an algebraic $\mathbb{C}$-scheme with $G$-action
and $\pi_Y \colon Y \to Y'$ a $G$-invariant morphism of 
algebraic $\mathbb{C}$-schemes 
where $G$ acts on $Y'$ 
trivially. 
Suppose that $Y'=\cup_{i \in I}V_i'$
 is an affine open cover such that 
$V_i=\pi_Y^{-1}(V_i')$ is affine 
and $\pi|_{V_i} \colon V_i \to V_i'$ is isomorphic to 
$V_i \to V_i \sslash G$.
Then for an analytic open subset $U \subset Y'$ 
and a $G$-invariant closed analytic subspace 
$Z \subset \pi_Y^{-1}(U)$, the analytic 
Hilbert quotient $Z\sslash G$ exists as a 
closed analytic subspace of $U$. 
\end{lem}
\begin{proof}
Let $U_i=U \cap V_i'$
and $Z_i=Z \cap V_i$. 
By applying Lemma~\ref{lem:Zquot} for 
$Z_i \subset \pi_Y^{-1}(U_i) \subset V_i$, 
we obtain the analytic Hilbert quotient 
$Z_i \sslash G \subset U_i$. 
By the construction, they glue to give 
a desired analytic Hilbert quotient
$Z\sslash G \subset U$.
\end{proof}
\begin{rmk}\label{rmk:Zquot2}
The situation of Lemma~\ref{lem:Zquot2}
happens for a GIT quotient of 
semistable locus w.r.t. a 
$G$-linearization on a quasi-projective scheme. 
\end{rmk}
We next discuss the universality of
analytic Hilbert quotients:
\begin{defi}\label{univ}
An analytic Hilbert quotient (\ref{AHilb}) satisfies the 
universality if 
for any $G$-invariant 
analytic map $h \colon Z \to Z'$ to a complex analytic space $Z'$, there 
is a unique factorization 
\begin{align}\label{univ:quot}
h \colon Z \stackrel{\pi_Z}{\to} Z \sslash G \to Z'.
\end{align}
\end{defi}
The above universality is proved in~\cite[Corollary~4]{MR1103041} when 
$Z$ is a reduced Stein space and $Z'=\mathbb{C}^n$. 
Below show the universality for the analytic 
Hilbert
quotients given in Lemma~\ref{lem:Zquot2}.
We prepare the following lemma: 
\begin{lem}\label{lem:univ:prepare}
Let $\pi_Z \colon Z \to Z\sslash G$
be the analytic Hilbert quotient given 
in Lemma~\ref{lem:Zquot2}. 
Then for any family of $G$-invariant
closed (not necessary analytic)
subsets $\{W_{\lambda}\}_{\lambda \in \Lambda}$
in $Z$, the image $\pi_Z(W_{\lambda})$ is closed 
in $Z\sslash G$ and we have the identity
\begin{align}\label{image:id}
\pi_Z \left(\bigcap_{\lambda \in \Lambda} W_{\lambda} \right)=
\bigcap_{\lambda \in \Lambda}\pi_Z(W_{\lambda}).
\end{align}
\end{lem}
\begin{proof}
The question is local on $Z\sslash G$, so we may assume that 
$Y$ is affine and $Y'=Y\sslash G$. 
Since $Z, Z\sslash G$ are closed in $\pi_Y^{-1}(U)$, $U$, 
we may also assume that $Z=\pi_Y^{-1}(U)$, $Z\sslash G=U$. 
Let $S \subset Y$ be a Kempf-Ness set as in Theorem~\ref{thm:KN}. 
Then for $S'\cneq \pi_Y^{-1}(U) \cap S$, we have 
the homeomorphism $S'/K \stackrel{\cong}{\to}U$. 
Therefore 
for $W_{\lambda}' \cneq S' \cap W_{\lambda}$, we have 
$\pi_Z(W_{\lambda}')=\pi_Z(W_{\lambda})$. 
Since each 
$W_{\lambda}'$ is a $K$-invariant closed subset of 
$S'$, 
its image $\pi_{Z}(W_{\lambda}')$ is a closed subset of $U$
and the identity (\ref{image:id}) holds. 
\end{proof}

The desired universality is proved in the following lemma:  
\begin{lem}\label{lem:universal}
The analytic Hilbert quotient 
$\pi_Z \colon Z \to Z\sslash G$ 
in Lemma~\ref{lem:Zquot2} satisfies the universality in 
Definition~\ref{univ}. 
\end{lem}
\begin{proof}
Let $h \colon Z \to Z'$ be a $G$-invariant analytic map 
to a complex analytic space $Z'$. 
We take an open cover $Z'=\cup_{\lambda \in \Lambda}\uU_{\lambda}'$
such that $\uU_{\lambda}'$ is a closed analytic subspace 
of an open subset in $\mathbb{C}^n$. 
Let 
$W_{\lambda}' \cneq Z' \setminus \uU_{\lambda}'$
and $W_{\lambda} \cneq h^{-1}(W_{\lambda}')$. 
Then 
each $W_{\lambda}$ is a $G$-invariant closed 
subset of $Z$. By Lemma~\ref{lem:univ:prepare}, 
the image $\pi_Z(W_{\lambda}) \subset Z \sslash G$
is closed and 
\begin{align*}
\bigcap_{\lambda \in \Lambda}\pi_Z(W_{\lambda})
=\pi_Z \left(\bigcap_{\lambda \in \Lambda} W_{\lambda} \right)=
\pi_Z \circ h^{-1} \left(\bigcap_{\lambda \in \Lambda}
W_{\lambda}'\right) =\emptyset. 
\end{align*}
Here the last identity follows because 
$\{\uU_{\lambda}'\}_{\lambda \in \Lambda}$ is an open 
cover of $Z'$. 
It follows that by setting 
$\uU_{\lambda}\cneq (Z\sslash G) \setminus \pi_Z(W_{\lambda})$,
we have an open cover 
$Z\sslash G =\cup_{\lambda \in \Lambda} \uU_{\lambda}$ and the
diagram
\begin{align*}
\xymatrix{
\pi_Z^{-1}(\uU_{\lambda}) \ar@<-0.3ex>@{^{(}->}[r] \ar[d]_-{\pi_Z} &
h^{-1}(\uU_{\lambda}') \ar[d]_-{h} & \\
\uU_{\lambda} 	
\ar@{.>}[r] & \uU_{\lambda}' \ar@<-0.3ex>@{^{(}->}[r] & \mathbb{C}^n. 
}
\end{align*}
Here the top horizontal arrow is an open immersion, and 
the right horizontal arrow is a locally closed embedding. 
By the property (2) in Definition~\ref{def:Hquot}, 
there is a unique analytic map $\uU_{\lambda} \to \uU_{\lambda}'$
which makes the above diagram commutes. 
By the uniqueness, they glue to give a desired factorization 
(\ref{univ:quot}). 

\end{proof}
\subsection{Moduli spaces of representations of quivers with convergent relations}
We return to the situation of Section~\ref{subsec:conv}. 
\begin{defi}
A convergent relation $I$ of a quiver $Q$ is a
collection of 
finite number of elements 
\begin{align*}
I=(f_1, \ldots, f_l), \ f_i \in \mathbb{C}\{Q\}.
\end{align*}
\end{defi}
Using the lemmas in the previous subsection, we have the following: 
\begin{lem}\label{lem:Qconv}
Given a convergent relation $I=(f_1, \ldots, f_l)$
of a quiver $Q$ and its dimension vector $\vec{m}$, 
there is an analytic open neighborhood of $0$
\begin{align*}
0 \in V \subset M_{Q}(\vec{m})
\end{align*}
such that each $\Hom(V_a, V_b)$-valued 
formal function $f_i(a, b, \vec{m})$ defined by (\ref{Vab})
for $f=f_i$
absolutely converges on $\pi_Q^{-1}(V)$. Here $\pi_Q$ is the 
quotient map
\begin{align*}
\pi_Q \colon 
\mathrm{Rep}_Q(\vec{m}) \to M_{Q}(\vec{m}). 
\end{align*}
\end{lem}
\begin{proof}
Let $\uU$ be an open neighborhood of $0 \in \mathrm{Rep}_Q(\vec{m})$ as 
in (\ref{open:V}), where 
each $f_{i}(a, b, \vec{m})$ absolutely converges on $\uU$. 
Since for $g=(g_i)_{i\in V(Q)} \in G$ and 
$u=(u_e)_{e \in E(Q)}$ we have 
\begin{align*}
f_i(a, b, \vec{m})(g \cdot u)=g_b^{-1} \circ f_i(a, b, \vec{m})(u) \circ g_a
\end{align*}
the $\Hom(V_a, V_b)$-valued function 
$f_i(a, b, \vec{m})$ absolutely converges on 
$G \cdot \uU$. 
By Lemma~\ref{lem:saturated2}, 
there is a saturated open subset 
$0 \in \vV \subset G \cdot \uU$. 
Then by Lemma~\ref{lem:saturated}, 
$\vV=\pi_Q^{-1}(V)$ for 
an open subset $0 \in V \subset M_Q(\vec{m})$.  
\end{proof}
For a quiver $Q$ with a convergent relation $I=(f_1, \ldots, f_l)$, 
let $\vec{m}$ be its dimension vector and 
take an open subset $V \subset M_Q(\vec{m})$ as in Lemma~\ref{lem:Qconv}. 
By Lemma~\ref{lem:Qconv}, we have the $G$-invariant 
closed analytic subspace of 
$\pi_{Q}^{-1}(V)$
\begin{align}\label{Rep:V}
\mathrm{Rep}_{(Q, I)}(\vec{m})|_{V} \subset
\pi_Q^{-1}(V)
\end{align}
whose structure sheaf is given by
\begin{align*}
\oO_{\mathrm{Rep}_{(Q, I)}(\vec{m})|_{V}}
=\oO_{\pi_Q^{-1}(V)}/(f_i(a, b, \vec{m})_{jk}, 
a, b \in V(Q)).
\end{align*}
Here $f_i(a, b, \vec{m})_{jk}$ is the matrix component 
of 
the analytic map 
\begin{align*}
f_i(a, b, \vec{m}) \colon \pi_Q^{-1}(V) \to \Hom(V_a, V_b).
\end{align*}
By taking the quotient by $G$, we have the
following definition: 
\begin{defi}\label{defi:cmoduli}
Let $Q$ be a quiver with a convergent relation $I$, 
and $\vec{m}$ its dimension vector. 
Then for a sufficiently small 
analytic open neighborhood $0 \in V \subset M_Q(\vec{m})$, 
we define 
the complex analytic stack $\mM_{(Q, I)}(\vec{m})|_{V}$ and 
complex analytic space $M_{(Q, I)}(\vec{m})|_{V}$ by 
\begin{align*}
\mM_{(Q, I)}(\vec{m})|_{V} &\cneq [\mathrm{Rep}_{(Q, I)}(\vec{m})|_{V}/G], \\
M_{(Q, I)}(\vec{m})|_{V} 
&\cneq \mathrm{Rep}_{(Q, I)}(\vec{m})|_{V} \sslash G. 
\end{align*}
Here $\mathrm{Rep}_{(Q, I)}(\vec{m})|_{V} \sslash G$
is the analytic Hilbert quotient of $\mathrm{Rep}_{(Q, I)}(\vec{m})|_{V}$, 
given in Lemma~\ref{lem:Zquot}. 
\end{defi}

\subsection{Convergent super-potential}\label{subsec:potential}
For a quiver $Q$, its convergent super-potential 
is defined as follows. 
\begin{defi}\label{def:conv:pot}
A convergent super-potential of a quiver $Q$ is an element 
\begin{align*}
W \in \mathbb{C}\{ Q \}/[\mathbb{C}\{ Q \}, \mathbb{C}\{ Q \}]. 
\end{align*}
\end{defi}  
A convergent super-potential $W$ of $Q$ is represented by 
a formal sum
\begin{align}\notag
W=\sum_{n\ge 1}
\sum_{\begin{subarray}{c}
\{1, \ldots, n+1\} \stackrel{\psi}{\to} V(Q), \\
\psi(n+1)=\psi(1)
\end{subarray}}
\sum_{e_i \in E_{\psi(i), \psi(i+1)}}
a_{\psi, e_{\bullet}} \cdot e_1 e_2\ldots e_{n}
\end{align}
with $\lvert a_{\psi, e_{\bullet}} \rvert <C^n$
for a constant $C>0$. 

For $i, j \in V(Q)$, let 
$\mathbf{E}_{i, j}$ be the $\mathbb{C}$-vector space 
spanned by $E_{i, j}$. 
We set
\begin{align}\label{e:dual}
E_{i, j}^{\vee} \cneq \{ e^{\vee} : 
e \in E_{i, j}\} \subset \mathbf{E}_{i, j}^{\vee}. 
\end{align}
Here for $e \in E_{i, j}$, 
the element $e^{\vee} \in \mathbf{E}_{i, j}^{\vee}$
is defined by the condition
$e^{\vee}(e)=1$
and $e^{\vee}(e')=0 $
for any $e\neq e' \in E_{i, j}$, i.e. 
$E_{i, j}^{\vee}$ is the dual basis of $E_{i, j}$. 
For a map $\psi \colon \{1, \ldots, n+1\} \to V(Q)$
with $\psi(1)=\psi(n+1)$
and elements $e_i \in E_{\psi(i), \psi(i+1)}$, 
$e\in E(Q)$, 
we set
\begin{align*}
\partial_{e^{\vee}}(e_1 \ldots e_n)
=\sum_{a=1}^{n} e^{\vee}(e_a) e_{a+1} \ldots e_n e_1 \ldots e_{a-1}. 
\end{align*}
Here $e^{\vee}(e_a)=0$ if 
$(s(e_a), t(e_a)) \neq (s(e), t(e))$. 
The above partial differential extends to a linear map
\begin{align*}
\partial_{e^{\vee}} \colon 
 \mathbb{C}\{ Q \}/[\mathbb{C}\{ Q \}, \mathbb{C}\{ Q \}]
\to \mathbb{C}\{ Q \}. 
\end{align*}
For a convergent super-potential $W$, 
the set of elements in 
$\mathbb{C}\{Q\}$
\begin{align*}
\partial W \cneq \{ \partial_{e^{\vee}}W : 
e \in E(Q) \}
\end{align*}
is a convergent relation of $Q$. 

For a dimension vector $\vec{m}$
of $Q$, let $\tr W$ be the formal function 
of $u=(u_e)_{e\in E(Q)} \in \mathrm{Rep}_Q(\vec{m})$ defined by
\begin{align*}
\tr W(u) \cneq \sum_{n\ge 1}
\sum_{\begin{subarray}{c}
\{1, \ldots, n+1\} \stackrel{\psi}{\to} V(Q), \\
\psi(n+1)=\psi(1)
\end{subarray}}
\sum_{e_i \in E_{\psi(i), \psi(i+1)}}
a_{\psi, e_{\bullet}} \cdot \tr(u_{e_n} \circ u_{e_{n-1}} 
\circ \cdots \circ u_{e_1}).
\end{align*}
The above formal function on 
$\mathrm{Rep}_Q(\vec{m})$ is $G$-invariant. 
By the argument of Lemma~\ref{lem:Qconv}, 
there is an 
analytic open neighborhood 
$0 \in V \subset M_Q(\vec{m})$
such that the formal function 
$\tr W$
absolutely converges on 
$\pi_Q^{-1}(V)$ to give a 
$G$-invariant holomorphic function
\begin{align*}
\tr W \colon \pi_Q^{-1}(V) \to \mathbb{C}. 
\end{align*}
Then for the relation $I=\partial W$, 
it is easy to see (and well-known when $W$ is a usual 
super-potential of $Q$)
that 
the analytic subspace (\ref{Rep:V}) 
equals to the critical locus of $\tr W$ in $\pi_Q^{-1}(V)$: 
\begin{align*}
\mathrm{Rep}_{(Q, \partial W)}(\vec{m})|_{V}=\{ d(\tr W)=0\}. 
\end{align*}
In particular, we have
\begin{align*}
\mM_{(Q, \partial W)}(\vec{m})|_{V}=\left[\{ d(\tr W)=0\}/G \right]. 
\end{align*}

\section{Moduli stacks of semistable sheaves}\label{sec:moduli}
In this section, we recall some basic notions and facts 
on moduli spaces of semistable sheaves, whose details 
are available in~\cite{MR1450870}. 
Then we state the precise form of Theorem~\ref{intro:thm1}
in Theorem~\ref{thm:precise}. 
In what follows, we always assume that the varieties or
schemes are defined over $\mathbb{C}$. 
\subsection{Gieseker semistable sheaves}
Let 
\begin{align*}
(X, \oO_X(1))
\end{align*}
 be a polarized smooth projective variety
with $\omega=c_1(\oO_X(1))$. 
For a coherent sheaf $E$ on $X$, its \textit{Hilbert polynomial} is 
defined by
\begin{align*}
\chi(E \otimes \oO_X(m))=a_d m^d+a_{d-1}m^{d-1}+\cdots
\end{align*}
where $d=\dim \Supp(E)$ and $a_d$ is a positive rational number. 
The \textit{reduced Hilbert polynomial} is defined by
\begin{align*}
\overline{\chi}(E, m)
 \cneq \frac{\chi(E \otimes \oO_X(m))}{a_d} \in \mathbb{Q}[m]. 
\end{align*}
For polynomials $p_i(m) \in \mathbb{Q}[m]$ with $i=1, 2$, 
we write $p_1(m) \succ p_2(m)$
if $\deg p_1<\deg p_2$ or $\deg p_1=\deg p_2$, $p_1(m) >p_2(m)$ for $m\gg 0$. 
Then $(\mathbb{Q}[m], \succ)$ is an ordered set.

By definition, 
a coherent sheaf $E$ on $X$ is said to be 
\textit{$\omega$-Gieseker (semi)stable}
if for any non-zero subsheaf $E' \subsetneq E$, we have the inequality
\begin{align*}
\overline{\chi}(E', m) \prec (\preceq) \overline{\chi}(E, m).
\end{align*}
For any Gieseker semistable sheaf $E$ on $X$, 
it has a filtration 
(called \textit{J\"ordar-H\"older (JH) filtration})
\begin{align*}
0=F_0 \subset F_1 \subset F_2 \subset \cdots \subset F_k=E
\end{align*}
such that 
each $F_i/F_{i-1}$ is $\omega$-Gieseker stable 
whose reduced Hilbert polynomial coincides with $\overline{\chi}(E, m)$. 
The JH filtration is not necessary unique, but its 
subquotient
\begin{align*}
\gr(E) \cneq \bigoplus_{i=1}^k F_i/F_{i-1}
\end{align*}
is uniquely determined up to isomorphism. 
For two $\omega$-Gieseker semistable sheaves 
$E, E'$ on $X$, they
are called \textit{S-equivalent} if 
$\gr(E)$ and $\gr(E')$ are isomorphic. 

\subsection{Moduli spaces of semistable sheaves}\label{subsec:moduli}
Let 
$\mM$ be the 2-functor
\begin{align}\label{stack:M}
\mM \colon
\sS ch/\mathbb{C} \to \gG roupoid
\end{align}
which sends a $\mathbb{C}$-scheme $S$ to the 
groupoid of $S$-flat coherent sheaves on $X \times S$.  
The stack $\mM$ is an algebraic stack locally of finite type 
over $\mathbb{C}$. 
Let $\Gamma$ be the image of the Chern character map
\begin{align*}
\Gamma \cneq \Imm (\ch \colon K(X) \to H^{\ast}(X, \mathbb{Q})). 
\end{align*}
For each $v \in \Gamma$, we have an open substack
of finite type
\begin{align*}
\mM_{\omega}(v) \subset \mM
\end{align*}
consisting of flat families of $\omega$-Gieseker semistable sheaves
with Chern character $v$. 

The stack $\mM_{\omega}(v)$ is constructed as a global quotient 
stack of a quasi-projective scheme. 
For $[E] \in \mM_{\omega}(v)$, 
we take $m\gg 0$ and 
a vector space $\mathbf{V}$ satisfying
\begin{align*}
\dim \mathbf{V} =\chi(E(m))=\dim H^0(E(m)). 
\end{align*}
The above condition depends only on $v$, and independent of $E$
for $m \gg 0$. 
Let $\mathrm{Quot}(\mathbf{V}, v)$ be the Grothendieck Quot scheme 
parameterizing quotients
\begin{align}\label{quot:s}
s \colon \mathbf{V} \otimes \oO_X(-m) \twoheadrightarrow E
\end{align}
in $\Coh(X)$
with $\ch(E)=v$. 
Then there is an open subscheme 
\begin{align*}
\mathrm{Quot}^{\circ}(\mathbf{V}, v) \subset \mathrm{Quot}(\mathbf{V}, v)
\end{align*} parameterizing 
quotients (\ref{quot:s}) such that 
$E$ is $\omega$-Gieseker semistable and 
the induced linear map
$\mathbf{V} \to H^0(E(m))$
is an isomorphism.  
The algebraic group $\mathrm{GL}(\mathbf{V})$ acts on 
$\mathrm{Quot}^{\circ}(\mathbf{V}, v)$ by 
\begin{align*}
g \cdot (\mathbf{V} \otimes \oO_X(-m) \stackrel{s}{\twoheadrightarrow} E)
=(\mathbf{V} \otimes \oO_X(-m) \stackrel{s \circ g}{\twoheadrightarrow} E)
\end{align*}
and the stack $\mM_{\omega}(v)$ is 
described as
\begin{align*}
\mM_{\omega}(v)=[\mathrm{Quot}^{\circ}(\mathbf{V}, v)/\GL(\mathbf{V})]. 
\end{align*}

The above construction is compatible with the 
Geometric Invariant Theory (GIT). 
If we take the closure of $\mathrm{Quot}^{\circ}(\mathbf{V}, v)$,
\begin{align*}
\overline{\mathrm{Quot}}^{\circ}(\mathbf{V}, v) \subset \mathrm{Quot}(\mathbf{V}, v)
\end{align*}
then there is a $\GL(\mathbf{V})$-linearized polarization 
on $\overline{\mathrm{Quot}}^{\circ}(\mathbf{V}, v)$ such that 
its open locus 
$\mathrm{Quot}^{\circ}(\mathbf{V}, v)$ is the GIT semistable locus with respect 
to the above $\GL(\mathbf{V})$-linearized polarization. 
In particular, we have the good quotient
morphism (which is in particular a good moduli space
in the sense of~\cite{MR3237451})
\begin{align*}
p_M \colon \mM_{\omega}(v) \to M_{\omega}(v) \cneq \mathrm{Quot}^{\circ}(\mathbf{V}, v)\sslash \GL(\mathbf{V}).
\end{align*}
Namely, there is a $\GL(\mathbf{V})$-invariant affine 
open cover 
\begin{align*}
\mathrm{Quot}^{\circ}(\mathbf{V}, v)=\bigcup_i U_i, \ 
U_i=\Spec R_i
\end{align*}
such that 
$M_{\omega}(v)$ has the following affine open cover
\begin{align*}
M_{\omega}(v)=\bigcup_i U_i \sslash \GL(\mathbf{V}), \ 
U_i \sslash \GL(\mathbf{V})=\Spec R_i^{\GL(\mathbf{V})}. 
\end{align*}
By the GIT construction of $M_{\omega}(v)$, 
two points $x_1, x_2 \in \mathrm{Quot}^{\circ}(\mathbf{V})$ are mapped to the 
same point by $p_M$ if and only if 
their orbit closures intersect, i.e.
\begin{align*}
\overline{\GL(\mathbf{V}) \cdot x_1} \cap \overline{\GL(\mathbf{V}) \cdot x_2} \neq \emptyset. 
\end{align*}
It is also known that the above condition is equivalent to 
that, if $x_i$ corresponds to a $\omega$-Gieseker semistable sheaf $E_i$, 
then $E_1$ and $E_2$ are $S$-equivalent. 
In fact, the projective scheme $M_{\omega}(v)$ is the coarse moduli 
space of $S$-equivalence classes of 
$\omega$-Gieseker semistable sheaves with 
Chern character $v$. 
So every point $p \in M_{\omega}(v)$ is represented by 
a direct sum of $\omega$-Gieseker stable sheaves $E$ 
(called a \textit{polystable sheaf}), written as
\begin{align}\label{polystable}
E=\bigoplus_{i=1}^k V_i \otimes E_i. 
\end{align}
Here each $V_i$ is a finite dimensional vector space, 
$E_i$ is a $\omega$-Gieseker stable sheaf with 
$\overline{\chi}(E_i, m)=\overline{\chi}(E, m)$
for all $i$.

\subsection{Ext-quiver}\label{subsec:Extquiver}
Suppose that $E \in \Coh(X)$ is 
of the form (\ref{polystable}). 
Then the 
collection of 
the sheaves $(E_1, \ldots, E_k)$ forms 
a simple collection, defined below: 
\begin{defi}
A collection of coherent sheaves $(E_1, \ldots, E_k)$ 
is called a simple collection if 
$\Hom(E_i, E_j)=\mathbb{C} \cdot \delta_{ij}$. 
\end{defi}
Let $E_{\bullet}=(E_1, \ldots, E_k)$
be a simple collection of coherent sheaves on $X$. 
For each $1\le i, j \le k$, 
we fix a finite subset
\begin{align}
E_{i, j} \subset \Ext^1(E_i, E_j)^{\vee}
\end{align}
giving a basis of $\Ext^1(E_i, E_j)^{\vee}$. 
We
define the quiver $Q_{E_{\bullet}}$ as follows. 
The set of vertices and edges are given by 
\begin{align*}
V(Q_{E_{\bullet}})=\{1, 2, \ldots, k\}, \ 
E(Q_{E_{\bullet}})=\coprod_{1\le i, j \le k}
E_{i,j}. 
\end{align*}
The maps $s, t \colon 
E(Q_{E_{\bullet}}) \to V(Q_{E_{\bullet}})$
are given by 
\begin{align*}
s|_{E_{i, j}}=i, \ t|_{E_{i, j}}=j. 
\end{align*}
The resulting quiver $Q_{E_{\bullet}}$ is called 
the \textit{Ext-quiver}
of $E_{\bullet}$. 

We can now state the precise statement of 
Theorem~\ref{intro:thm1}: 
\begin{thm}\label{thm:precise}
Let $X$ be a smooth projective variety, 
and let $\mM_{\omega}(v)$ be the 
moduli stack of $\omega$-Gieseker semistable sheaves on $X$ with 
Chern character $v$. We have 
the natural morphism to its coarse moduli space
\begin{align*}
p_{M} \colon \mM_{\omega}(v) \to M_{\omega}(v). 
\end{align*}  
For $p \in M_{\omega}(v)$, it is represented by a
sheaf $E$ of the form
\begin{align*}
E=\bigoplus_{i=1}^k V_i \otimes E_i
\end{align*}
where $E_{\bullet}=(E_1, \ldots, E_k)$ is a simple collection. 
Let $Q_{E_{\bullet}}$ be the corresponding Ext-quiver 
and $\vec{m}$ its dimension vector given by 
$\vec{m}=(m_1, \ldots, m_k)$, where  
$m_i=\dim V_i$. 
Then there is a
convergent relation $I_{E_{\bullet}}$ of $Q_{E_{\bullet}}$, 
 analytic open neighborhoods 
$p \in U \subset M_{\omega}(v)$, 
$0 \in V \subset M_{Q_{E_{\bullet}}}(\vec{m})$
and commutative isomorphisms
\begin{align}\label{dia:comiso}
\xymatrix{
\mM_{(Q_{E_{\bullet}, I_{E_{\bullet}}})}(\vec{m})|_{V} 
\ar[r]^-{\cong}
\ar[d]_-{p_Q} & p_M^{-1}(U)  \ar[d]^-{p_M} \\
M_{(Q_{E_{\bullet}, I_{E_{\bullet}}})}(\vec{m})|_{V}
\ar[r]^-{\cong} & 
U. 
}
\end{align}
Here the bottom arrow sends $0$ to $p$. 
\end{thm}
The proof of Theorem~\ref{thm:precise} will 
be completed in Proposition~\ref{prop:complete} below.

\section{Deformations of coherent sheaves}\label{sec:deform}
In this section, we describe deformation theory of 
coherent sheaves via dg-algebras and their  
minimal 
$A_{\infty}$-models. 
The arguments are already known for vector bundles~\cite{MR1950958, JuTu}
and we apply similar arguments for resolutions of 
coherent sheaves by vector bundles. 

The above description will give local atlas of the 
moduli stack $\mM$ in Subsecton~\ref{subsec:moduli}
via finite dimensional $A_{\infty}$-algebras. 
More precisely for a given 
coherent sheaf $E$ on a smooth projective variety $X$, 
we compare the following three descriptions of the
deformation space of $E$: 
\begin{enumerate}
\item An open neighborhood of the algebraic stack $\mM$
given in Subsection~\ref{subsec:moduli} at the point $[E] \in \mM$. 

\item The Mauer-Cartan locus associated with the infinite 
dimensional dg-algebra $\RHom(E, E)$. 

\item The Mauer-Cartan locus associated with the 
finite dimensional minimal $A_{\infty}$-algebra
$\Ext^{\ast}(E, E)$. 
\end{enumerate}
We will compare the above descriptions by 
first constructing the map $(3) \Rightarrow (2)$
in Lemma~\ref{prop:restrict}. 
Then we will construct a map $(2) \Rightarrow (1)$, 
and then composing we get a 
desired atlas $(3) \Rightarrow (1)$
in Proposition~\ref{prop:rest2}.

\subsection{Deformations of vector bundles}\label{subsec:dg}
We recall some basic facts on the 
deformation theory of vector bundles via gauge theory, 
and 
fix some notation (see~\cite{MR1950958} for details). 
For a holomorphic vector bundle $\eE \to X$ on 
a smooth projective variety 
$X$, we denote by $\aA^{p, q}(\eE)$ the
sheaf of $\eE$-valued
$(p, q)$-forms on $X$, 
and set
\begin{align*}
A^{p, q}(\eE) \cneq \Gamma(X, \aA^{p, q}(\eE)). 
\end{align*} 
The holomorphic structure on $\eE$ 
is given by the Dolbeaut connection
\begin{align*}
\overline{\partial}_{\eE} \colon \aA^{0, 0}(\eE) \to \aA^{0, 1}(\eE).  
\end{align*}
The Dolbeaut connection 
extends to the Dolbeaut complex
\begin{align*}
0 \to 
\aA^{0, 0}(\eE) \to \aA^{0, 1}(\eE) \to \cdots 
\to \aA^{0, i}(\eE) \to \aA^{0, i+1}(\eE) \to \cdots
\end{align*}
giving a resolution of $\eE$. 
The complex $\aA^{0, \ast}(\eE)$ is an elliptic 
complex 
(see~\cite[Chapter~IV, Section~5]{MR0515872}), 
whose global section computes 
$H^{\ast}(X, \eE)$, i.e. 
\begin{align*}
H^{k}(X, \eE)=\hH^{k}(A^{0, \ast}(\eE)). 
\end{align*}
Any other holomorphic structure on $\eE$ is given by 
the Dolbeaut connection 
of the form 
\begin{align*}
\overline{\partial}_{\eE}+A
\colon \aA^{0, 0}(\eE) \to \aA^{0, 1}(\eE)
\end{align*}
for some $A \in A^{0, 1}(\eE nd(\eE))$. Conversely 
given $A \in A^{0, 1}(\eE nd(\eE))$, the 
connection $\overline{\partial}_{\eE}+A$ 
gives a holomorphic structure on $\eE$ if and only if 
its square is zero, i.e. 
\begin{align*}
\mathrm{ad}(\overline{\partial}_{\eE})(A)+A \circ A=0. 
\end{align*}
The above equation is the Mauer-Cartan (MC) equation of the 
dg-algebra
\begin{align}\label{g:vect}
\mathfrak{g}_{\eE}^{\ast} \cneq A^{0, \ast}(\eE nd(\eE)). 
\end{align}
The quotient of the solution space of the 
MC equation of $\mathfrak{g}_{\eE}^{\ast}$
by the gauge group 
of $\cC^{\infty}$-automorphisms of $\eE$
 describes the 
deformation space of $\eE$ as holomorphic 
vector bundles.

\subsection{Deformations of complexes}\label{subsec:complex}
We have a similar deformation theory for complexes of 
vector bundles. 
Let 
\begin{align}\label{seq:E0}
\eE^{\bullet}=(\cdots \to 0 \to \eE^i \stackrel{d^i}{\to} \eE^{i+1} \to \cdots \to \eE^j \to 0 
\to \cdots)
\end{align}
be a bounded complex of
holomorphic vector 
bundles on $X$.
By 
taking the Dolbeaut complex
$\aA^{0, \ast}(\eE^i)$ for each 
$\eE^i$, we obtain the double complex
$\aA^{0, \ast}(\eE^{\bullet})$. 
Let $\mathrm{Tot}(-)$ means the total complex of the double complex. 
 We set
\begin{align}\label{A:tot}
A^{0, \ast}(\eE^{\bullet}) \cneq
\mathrm{Tot}(\Gamma(X, \aA^{0, \ast}(\eE^{\bullet}))). 
\end{align}
Similarly to the vector bundle case, 
the complex $\mathrm{Tot}(\aA^{0, \ast}(\eE^{\bullet}))$
is elliptic, and 
its global section computes 
the hyper cohomology of $\eE^{\bullet}$
\begin{align}\label{compute:hyper}
\hH^k(\dR \Gamma(X, \eE^{\bullet}))=\hH^k(A^{0, \ast}(\eE^{\bullet})). 
\end{align}

Applying the construction (\ref{A:tot})
to the inner $\hH om$ complex
$\hH om^{\ast}(\eE^{\bullet}, \eE^{\bullet})$, 
we obtain the complex
\begin{align*}
\mathfrak{g}_{\eE^{\bullet}}^{\ast}
\cneq 
A^{0, \ast}(\hH om^{\ast}(\eE^{\bullet}, \eE^{\bullet})). 
\end{align*}
Its degree $k$ part is given by
\begin{align}\label{g:degk}
\mathfrak{g}_{\eE^{\bullet}}^k=\bigoplus_{p+q=k} \prod_{i} 
A^{0, q}(\hH om(\eE^i, \eE^{i+p}))
\end{align}
and the differential $d_{\mathfrak{g}}$ is induced 
by the Dolbeaut connections 
$\overline{\partial}_{\eE_i}$ on each $\eE_i$
together with 
the differentials $d^{\ast}$ in (\ref{seq:E0}). 
Also the composition
\begin{align*}
A^{0, q}(\hH om(\eE^i, \eE^{i+p})) \times
A^{0, q'}&(\hH om(\eE^{i+p}, \eE^{i+p+p'})) \\ 
&\to 
A^{0, q+q'}(\hH om(\eE^{i}, \eE^{i+p+p'}))
\end{align*}
defines 
 the product structure $\cdot$ on $\mathfrak{g}_{\eE^{\bullet}}^{\ast}$. 
Then 
it is straightforward to check that 
the data
\begin{align}\label{g:E}
(\mathfrak{g}_{\eE^{\bullet}}^{\ast}, d_{\mathfrak{g}}, \cdot)
\end{align}
is a dg-algebra.   

Let $\mathfrak{mc}$ be
the map defined by 
\begin{align*}
\mathfrak{mc} \colon \mathfrak{g}_{\eE^{\bullet}}^1 \to \mathfrak{g}_{\eE^{\bullet}}^2, \ 
\alpha \mapsto d_{\mathfrak{g}}(\alpha)+ \alpha \cdot \alpha.
\end{align*}
Its zero set 
\begin{align}\label{sol:MCeq}
\mathrm{MC}(\mathfrak{g}_{\eE^{\bullet}}^{\ast})=\{ 
\alpha \in \mathfrak{g}_{\eE^{\bullet}}^1 : \mathfrak{mc}(\alpha)=0\} 
\end{align}
is the solution of the Mauer-Cartan equation of the dg-algebra
$\mathfrak{g}_{\eE^{\bullet}}^{\ast}$. 
Note that 
an element $\alpha \in \mathfrak{g}_{\eE^{\bullet}}^1$ satisfies the 
MC equation iff 
\begin{align*}
(d_{\aA^{0, \ast}(\eE^{\bullet})}+\alpha)^2=0
\end{align*}
on $\aA^{0, \ast}(\eE^{\bullet})$. 
In this case, 
the data
\begin{align}\label{deform:E}
(\aA^{0, \ast}(\eE^{\bullet}), d_{\aA^{0, \ast}(\eE^{\bullet})}+\alpha)
\end{align}
determines 
a dg-$\aA^{0, \ast}(\oO_X)$-module. 
Then (\ref{deform:E}) is  
 a bounded complex of 
$\oO_X$-modules whose cohomologies 
are coherent (see~\cite[Lemma~4.1.5]{MR2648899}),
giving a deformation 
of the complex (\ref{seq:E0}) in the derived category. 

More explicitly, by 
(\ref{g:degk}) 
an element $\alpha \in \mathfrak{g}_{\eE^{\bullet}}^1$ consists of data
\begin{align}\label{write:alpha}
\alpha=(\alpha_0^i, \alpha_1^i, \alpha_2^i, \ldots), \ 
\alpha_{j}^i \in A^{0, j}(\hH om(\eE^i, \eE^{i-j+1}))
\end{align}
Suppose that the above $\alpha$ satisfies the 
MC equation $\mathfrak{mc}(\alpha)=0$. 
Then the diagram 
\begin{align*}
\xymatrix{
\cdots \ar[r] &
\aA^{0, 0}(\eE^{i-1}) \ar[r]^{} \ar[d] &
 \aA^{0, 0}(\eE^i) \ar[r]^{d^i+\alpha_0^i}
\ar[d]_{\overline{\partial}_{\eE^{i}}+\alpha_1^i} & \aA^{0, 0}(\eE^{i+1}) \ar[r]\ar[d] & \cdots \\
\cdots \ar[r] &
 \aA^{0, 1}(\eE^{i-1}) \ar[r]\ar[d]  &
 \aA^{0, 1}(\eE^i) \ar[r]\ar[d]_{\overline{\partial}_{\eE^{i}}+\alpha_1^i}
 & \aA^{0, 1}(\eE^{i+1}) 
\ar[r]\ar[d] & \cdots \\
\cdots \ar[r] &
\aA^{0, 2}(\eE^{i-1})\ar[r]  &
\aA^{0, 2}(\eE^i) \ar[r] & 
\aA^{0, 2}(\eE^{i+1}) \ar[r] & \cdots 
}
\end{align*}
satisfies the following: 
it is a complex in the horizontal direction, 
each square is commutative, and 
the compositions of vertical arrows are homotopic to 
zero with homotopy given by $\alpha_2^i$. 

In particular if $\alpha_j^i=0$ for $j\ge 2$,
then the above diagram 
extends to a double complex.
In this case 
\begin{align*}
\eE^i_{\alpha}=(\aA^{0, 0}(\eE^i), \overline{\partial}_{\eE^{i}}+\alpha_1^i)
\end{align*}
is a holomorphic structure on 
$\eE^i$. 
By setting 
\begin{align*}
d_{\alpha}^i=d^i+\alpha_0^i \colon \aA^{0, 0}(\eE^i) \to \aA^{0, 0}(\eE^{i+1})
\end{align*}
we have the bounded 
complex of holomorphic vector bundles on $X$
\begin{align}\label{E:alpha}
\cdots \to 0 \to \eE^{-n}_{\alpha} \stackrel{d^{-n}_{\alpha}}{\to}
  \cdots \to \eE^{-1}_{\alpha}
 \stackrel{d^{-1}_{\alpha}}{\to} \eE^0_{\alpha} \to 0  \to \cdots
\end{align}
giving a deformation of $\eE^{\bullet}$ as complexes.
Conversely given a deformation of $\eE^{\bullet}$ as
a complex, then it gives rise to the 
solution of MC equation of the form 
$\alpha=(\alpha_0^i, \alpha_1^i, 0, \ldots)$.  

For $\alpha, \alpha' \in \mathrm{MC}(\mathfrak{g}_{\eE^{\bullet}}^{\ast})$, 
$\alpha$ and $\alpha'$ are called \textit{gauge equivalent} if 
there exist 
\begin{align*}
\gamma=\{(\gamma_0^i, \gamma_1^i, \gamma_2^i, \ldots)\}_{i} \in 
\mathfrak{g}_{\eE^{\bullet}}^0, \ 
\gamma_j^i \in A^{0, j}(\hH om(\eE^i, \eE^{i-j}))
\end{align*}
where $\gamma_0^i$ 
gives an isomorphism $\eE^i \stackrel{\cong}{\to} \eE^i$
as $\cC^{\infty}$-vector bundles, such that 
we have
\begin{align}\label{isom:gauge}
\gamma \circ (d_{\aA^{0, \ast}(\eE^{\bullet})}+\alpha) \circ \gamma^{-1}=
d_{\aA^{0, \ast}(\eE^{\bullet})}+\alpha'.
\end{align}
In this case, we have the isomorphism of 
the dg-$\aA^{0, \ast}(\oO_X)$-modules
\begin{align*}
\gamma \colon 
(\aA^{0, \ast}(\eE^{\bullet}), d_{\aA^{0, \ast}(\eE^{\bullet})}+\alpha) \stackrel{\cong}{\to}
(\aA^{0, \ast}(\eE^{\bullet}), d_{\aA^{0, \ast}(\eE^{\bullet})}+\alpha')
\end{align*} 
giving isomorphic deformations of (\ref{seq:E0}) 
in the derived category. 

Suppose that the complex (\ref{seq:E0}) is quasi-isomorphic to 
a coherent sheaf $E$. Let $\mathrm{Def}_E$ be the deformation functor
\begin{align*}
\mathrm{Def}_E \colon \aA rt \to \sS et
\end{align*}
sending a finite dimensional commutative local $\mathbb{C}$-algebra 
$(R, \mathbf{m})$
to the set of isomorphism classes of $R$-flat deformation 
of $E$ to $X \times \Spec R$. 
Then it is shown in~\cite[Section~8]{DDE} that we have the functorial 
isomorphism
\begin{align*}
\mathrm{MC}(\mathfrak{g}_{\eE^{\bullet}}^{\ast} \otimes \mathbf{m})
/(\mbox{gauge equivalence}) \stackrel{\cong}{\to}
\mathrm{Def}_E(R) 
\end{align*}
by sending a solution of the MC equation to the 
cohomology of the 
corresponding deformation (\ref{deform:E}).

\subsection{Resolutions of coherent sheaves}
For a smooth projective variety $X$, 
we consider deformation theory of a sheaf
\begin{align*}
E \in \Coh(X)
\end{align*}
in terms of dg-algebra. 
As we recalled in Section~\ref{subsec:dg}, 
when $E$ is a vector bundle 
its deformation theory is described in terms of the 
dg-algebra (\ref{g:vect}). 
In general, we take a resolution of $E$ by 
vector bundles and consider 
the associated dg-algebra (\ref{g:E}). 

We first fix a resolution of $E$ by vector bundles 
in the following way. 
Let $\oO_X(1)$ be an ample line bundle on $X$. 
Then for $m_0 \gg 0$ we have the surjection
\begin{align*}
H^0(E(m_0)) \otimes \oO_{X}(-m_0) \twoheadrightarrow E. 
\end{align*}
Applying this construction to the kernel
of the above morphism and repeating, we obtain the 
resolution of $E$ of the form
\begin{align*}
\cdots \to W^i \otimes \oO_X(-m_i) \stackrel{d^i}{\to}
&W^{i+1} \otimes \oO_X(-m_{i+1}) \to
\cdots \\
&\cdots \to W^0 \otimes \oO_X(-m_0) \to E \to 0
\end{align*}
for finite dimensional vector spaces $W^i$. 
Since $X$ is smooth,
the kernel of $d^i$ for $i=-N$ with $N\gg 0$ is a 
vector bundle on $X$. 
Therefore we obtain the bounded resolution of 
$E$
\begin{align}\label{seq:E}
0 \to \eE^{-N} \stackrel{d^{-N}}{\to}
  \cdots \to \eE^{-1} \stackrel{d^{-1}}{\to} \eE^0 \to E \to 0
\end{align}
where $\eE^{-N}=\Ker(d^{-N})$ and 
$\eE^i=W^i \otimes \oO_X(-m_i)$ for 
$-N<i\le 0$. 

By replacing $m_i$ and $n$ if necessary, the above 
construction can be extended to local universal family 
of deformations of $E$. 
Let $\mM$ be the stack (\ref{stack:M}), and
take its local atlas 
\begin{align}\label{atlas}
(A, p) \to (\mM, [E])
\end{align}
 at $[E] \in \mM$, 
such that $A$ is a finite type affine scheme and a point 
$p \in A$ is sent to $[E]$. 
Let
\begin{align*}
E_{A} \in \Coh(X \times A)
\end{align*}
be the universal family. 
Let $\oO_{X \times A}(1)$ be the pull-back of 
$\oO_X(1)$ to $X \times A$. 
For $m_0 \gg 0$, 
the $\oO_A$-module 
$H^0(E_A(-m_0))$ is locally free of finite rank and 
we have the surjection
\begin{align*}
H^0(E_A(m_0)) \otimes_{\oO_A}\oO_{X \times A}(-m_0) \twoheadrightarrow E_A. 
\end{align*}
Similarly as above, we obtain the 
resolution of $E_A$ of the form
\begin{align*}
\cdots \to \wW^i \otimes_{\oO_A} \oO_{X \times A}(-m_i) \to
&\wW^{i+1} \otimes_{\oO_A} \oO_{X\times A}(-m_{i+1}) \to
\cdots \\
&\cdots \to W^0 \otimes_{\oO_A} \oO_{X\times A}(-m_0) \to E_A \to 0
\end{align*}
for locally free $\oO_A$-modules $\wW^i$ of finite rank. 
By taking the kernel at $i=-N$ for $N\gg 0$, we obtain 
the resolution of $E_A$
\begin{align}\label{seq:Eu}
0 \to \eE^{-N}_{A} \to
  \cdots \to \eE_A^{-1} \to \eE_A^0 \to E_{A} \to 0.
\end{align}
For $N\gg 0$, 
each $\eE^i_A$ is a vector bundle on $X \times A$, 
since  $E_A$ is a $A$-flat perfect object. 
By restricting it to 
$X \times \{p\}$, we obtain the 
resolution (\ref{seq:E}). 

\subsection{Minimal $A_{\infty}$-algebras}\label{subsec:minimal}
For a coherent sheaf $E$ on $X$, we fix a resolution 
$\eE^{\bullet}$
as in (\ref{seq:E})
and consider the dg-algebra (\ref{g:E})
\begin{align}\label{minimal:g}
\mathfrak{g}_E^{\ast} \cneq \mathfrak{g}_{\eE^{\bullet}}^{\ast}. 
\end{align}
When $E$ is a vector bundle, we just take the 
dg-algebra (\ref{g:vect}) in the argument below.  
By (\ref{compute:hyper}) we have 
\begin{align*}
\Ext^k(E, E) =\hH^k(\mathfrak{g}_E^{\ast}). 
\end{align*}
By the homological transfer theorem, there exists a 
minimal $A_{\infty}$-algebra structure $\{m_n\}_{n\ge 2}$
on $\Ext^{\ast}(E, E)$, 
and a
quasi-isomorphism 
\begin{align}\label{quasi:I}
I \colon (\Ext^{\ast}(E, E), \{m_n\}_{n\ge 2})
 \to (\mathfrak{g}_E^{\ast}, d_{\mathfrak{g}}, \cdot)
\end{align}
as $A_{\infty}$-algebras. 
Here the $A_{\infty}$-structure on $\Ext^{\ast}(E, E)$
consists of linear maps
\begin{align}\label{def:mn}
m_n \colon \Ext^{\ast}(E, E) \to \Ext^{\ast+2-n}(E, E), \ 
n\ge 2
\end{align}
 and the quasi-isomorphism (\ref{quasi:I})
is a collection of linear maps
\begin{align*}
I_n \colon \Ext^{\ast}(E, E)^{\otimes n} \to \mathfrak{g}_E^{\ast+1-n}.
\end{align*}
Both of $m_n$ and $I_n$ satisfy the $A_{\infty}$-constraints. 
The maps $m_n$ and $I_n$ are explicitly 
described in terms of Kontsevich-Soibelman's tree formula~\cite{MR1882331}
given as follows. 

Let us choose a K\"ahler metric on $X$, Hermitian metrics
on vector bundles $\eE^i$, and fix them. 
A standard argument in Hodge theory for elliptic complexes (for example, see~\cite{MR0515872})
yields linear embedding
\begin{align*}
i \colon \Ext^{\ast}(E, E) \hookrightarrow \mathfrak{g}_E^{\ast}
\end{align*}
which identifies 
$\Ext^{\ast}(E, E)$ with 
$\Delta=0$ where $\Delta$ is 
the Laplacian operator
\begin{align*}
\Delta=d_{\mathfrak{g}} d_{\mathfrak{g}}^{\ast}
+d_{\mathfrak{g}}^{\ast} d_{\mathfrak{g}} \colon 
\mathfrak{g}_{E}^{\ast} \to \mathfrak{g}_E^{\ast}.
\end{align*}
Here $d_{\mathfrak{g}}^{\ast}$ is the adjoint map of 
$d_{\mathfrak{g}}$
with respect to the above chosen K\"ahler metric on $X$ and Hermitian 
metrics on $\eE^i$. 
Moreover 
we have 
linear operators
\begin{align}\label{hodge}
p \colon 
\mathfrak{g}_E^{\ast}
\twoheadrightarrow \Ext^{\ast}(E, E), \ 
h \colon \mathfrak{g}_E^{\ast} \to \mathfrak{g}_E^{\ast-1}
\end{align}
satisfying the following relations
\begin{align}\label{relation}
p\circ i=\id, \ i \circ 
p=\id+ d_{\mathfrak{g}} \circ h + h\circ d_{\mathfrak{g}}. 
\end{align}
The homotopy operator $h$ is given by 
\begin{align}\label{h:homotopy}
h=-d_{\mathfrak{g}}^{\ast} \circ G
\end{align}
where $G$ is the Green's operator, which is an operator 
of order $-2$
(see~\cite[Chapter~IV]{MR0515872}),
hence $h$ is of order $-1$.

The $A_{\infty}$-product (\ref{def:mn})
is described by 
Kontsevich-Soibelman's tree formula as 
\begin{align}\label{KS:m}
m_n=\sum_{T \in \oO(n)} \pm m_{n, T}
\end{align}
where $\oO(n)$ is the set of isomorphism classes of binary 
rooted trees with $n$-leaves.
Here $m_{n, T}$ is the operation given by 
the composition associated to $T$, by 
putting $i$ on leaves, 
the product map $\cdot$ of $\mathfrak{g}_E^{\ast}$ on 
internal vertices, the homotopy $h$ on internal edges, and 
the projection $p$ on the root of $T$. 
For example, $m_3$ is given by 
\begin{align*}
m_3(x_1, x_2, x_3)=
\pm p(h(i(x_1) \cdot i(x_2)) \cdot i(x_3))
\pm p(i(x_1) \cdot h(i(x_1) \cdot i(x_2))). 
\end{align*}
The operation $I_n$ is similarly given by
\begin{align}\label{KS:I}
I_n=\sum_{T \in \oO(n)} \pm I_{n, T}
\end{align}
where $I_{n, T}$ is defined by replacing 
$p$ by $h$ in the construction of $m_{n, T}$. 
For example, $I_3$ is given by 
\begin{align*}
I_3(x_1, x_2, x_3)=
\pm h(h(i(x_1) \cdot i(x_2)) \cdot i(x_3))
\pm h(i(x_1) \cdot h(i(x_1) \cdot i(x_2))). 
\end{align*}
By~\cite[Appendix~A]{JuTu}, there exists 
another $A_{\infty}$-homomorphism
\begin{align}\label{P:inverse}
P \colon (\mathfrak{g}_E^{\ast}, d_{\mathfrak{g}}, \cdot) \to 
(\Ext^{\ast}(E, E), \{m_n\}_{n\ge 2})
\end{align}
which is a homotopy inverse of $I$, i.e. 
\begin{align*}
P \circ I=\id, \quad I \circ P \stackrel{\rm{homotopic}}{\sim}\id.
\end{align*}
Here two $A_{\infty}$-morphisms 
$f_1, f_2 \colon A_1 \to A_2$ 
between $A_{\infty}$-algebras $A_1$, $A_2$ 
are called \textit{homotopic} if there is an $A_{\infty}$-homomorphism
\begin{align*}
H \colon A_1 \to A_2 \otimes \Omega_{[0, 1]}^{\ast}
\end{align*}
such that 
$H(0)=f_1$ and $H(1)=f_2$, where 
$\Omega_{[0, 1]}^{\ast}$ is the dg-algebra of 
$\cC^{\infty}$-differential forms on the interval $[0, 1]$. 
The $A_{\infty}$-homomorphism $P$
consists of linear maps
\begin{align*}
P_n \colon (\mathfrak{g}_E^{\ast})^{\otimes n} \to 
\Ext^{\ast+1-n}(E, E)
\end{align*}
which are also described in terms of tree formula, whose details we omit (see~\cite[Appendix~A]{JuTu} for details). 

Later we will use some boundedness properties of linear maps
$m_n$, $I_n$ and $P_n$. 
Let us take an even number $l \gg 0$, 
e.g. $l>2\dim X$, and consider the 
Sobolev $(l, 2)$-norm 
$\lVert - \rVert_l$ 
on 
$\mathfrak{g}_E^{\ast}$. 
It also induces a norm $\lVert - \rVert_{l}$
on $\Ext^{\ast}(E, E)$
by the embedding $i$ in (\ref{hodge}). 
We denote by 
\begin{align*}
\mathfrak{g}_E^{\ast} \subset \widehat{\mathfrak{g}}_{E, l}^{\ast}
\end{align*}
the completion of $\mathfrak{g}_E^{\ast}$
with respect to the Sobolev norm 
$\lVert - \rVert_{l}$.

\begin{lem}\label{lem:mbound}
There is a constant $C>0$ independent of $n$
such that 
\begin{align*}
\lVert m_n \rVert_l<C^n, \ \lVert I_n \rVert_l<C^n, \ 
\lVert P_n \rVert_l<C^n. 
\end{align*}
Here $\lVert - \rVert_l$ for linear maps 
mean the operator norm with respect to the norm 
$\lVert - \rVert_l$ on $\mathfrak{g}_E^{\ast}$
or $\Ext^{\ast}(E, E)$. 
\end{lem}
\begin{proof}
When $E$ is a vector bundle, the lemma is proved 
in~\cite[Proposition~2.3.2]{MR1950958}
and~\cite[Lemma~A.1.1, Lemma~A.1.2, Lemma~A.1.5]{JuTu}.
The key ingredient of the proof is that 
the maps $m_n$, $I_n$, $P_n$ are constructed as in (\ref{KS:m}) using 
rooted trees, whose cardinality is bounded
as
\begin{align*}
\sharp \oO(n)=\frac{(2n-2)!}{(n-1)! n!} <4^{n-1},
\end{align*}
and the fact that the 
homotopy operator $h$, the product map on 
$\mathfrak{g}_E^{\ast}$ are extended to bounded operators 
\begin{align*}
\widehat{\mathfrak{g}}_{E, l}^{\ast} \stackrel{h}{\to}
 \widehat{\mathfrak{g}}_{E, l}^{\ast}, \quad 
\widehat{\mathfrak{g}}_{E, l}^{\ast} \times 
\widehat{\mathfrak{g}}_{E, l}^{\ast} \stackrel{\cdot}{\to} 
\widehat{\mathfrak{g}}_{E, l}^{\ast}. 
\end{align*}
When $E$ is a coherent sheaf which is not necessary a vector bundle, 
the above property still hold for the complex (\ref{seq:E}) without any modification: 
the boundedness of $h$ is a general fact for 
elliptic complexes
(see~\cite[Theorem~4.12]{MR0515872}),
as it is an operator of degree $-1$ given by (\ref{h:homotopy}),  
and that of the product $\cdot$
follows from our choice of $l\gg 0$
and a standard result of Sobolev spaces
(for example see~\cite[Theorem~25]{Willie}).  
Therefore the same argument for the vector bundle case 
proves the lemma. 
\end{proof}

\subsection{Deformations by $A_{\infty}$-algebras}
For $x \in \Ext^1(E, E)$, we consider the infinite series
\begin{align}\label{kappa}
\kappa(x)\cneq \sum_{n\ge 2}m_n(x, \ldots, x) 
\end{align}
where each term $m_n(x, \ldots, x)$ is an element of 
$\Ext^2(E, E)$. 
By Lemma~\ref{lem:mbound}, there is an analytic open neighborhood 
\begin{align}\label{open:U}
0 \in \uU \subset \Ext^1(E, E)
\end{align}
such that the series (\ref{kappa})
absolutely converges on $\uU$ to give a complex analytic morphism
\begin{align}\label{kappa:U}
\kappa \colon \uU \to \Ext^2(E, E). 
\end{align}
The equation $\kappa(x)=0$ is the Mauler-Cartan 
equation for the $A_{\infty}$-algebra (\ref{def:mn}). 
We set $T$ to be
\begin{align}\label{T:uU}
T \cneq \kappa^{-1}(0) \subset \uU
\end{align}
i.e. $T$ is the closed complex analytic subspace 
defined by the ideal of zero of the map (\ref{kappa:U}). 

On the other hand, for $x \in \Ext^1(E, E)$ we also consider the infinite 
series
\begin{align}\label{series:I}
I_{\ast}(x) \cneq \sum_{n\ge 1}I_n(x, \ldots, x) 
\end{align}
where each term $I_n(x, \ldots, x)$ is an element of
$\mathfrak{g}_E^1$. 
By Lemma~\ref{lem:mbound},
for sufficiently small open subset (\ref{open:U})
the series (\ref{series:I})
absolutely converges on $\uU$ to give a 
morphism of Banach analytic spaces
\begin{align}\label{I:analytic}
I_{\ast} \colon \uU \to \widehat{\mathfrak{g}}_{E, l}^1. 
\end{align}
\begin{lem}\label{prop:restrict}
The morphism (\ref{I:analytic}) restricts to the morphism of 
Banach analytic spaces
\begin{align}\label{restrict:Banach}
I_{\ast} \colon T\to \mathrm{MC}(\mathfrak{g}_E^{\ast}). 
\end{align}
Here $\mathrm{MC}(\mathfrak{g}_E^{\ast})$ is the solution of the 
Mauer-Cartan equation (\ref{sol:MCeq})
of the dg-algebra $\mathfrak{g}_E^{\ast}$. 
\end{lem}
\begin{proof}
The result is proved in~\cite[Section~2.2, Lemma~A.1.3]{JuTu}
when $E$ is a vector bundle, and the same 
argument applies for the complex (\ref{seq:E}). 
Since $I_{\ast}$ is an $A_{\infty}$-homomorphism, 
it preserves the MC locus, so it 
sends $T$ to $\mathrm{MC}(\widehat{\mathfrak{g}}_{E, l}^{\ast})$. 
For $x \in T$, the smoothness of $I_{\ast}(x)$ follows 
along with the argument of~\cite[Lemma~A.1.3]{JuTu}, 
by replacing $\overline{\partial}$ in \textit{loc.cit.} by 
the differential $d_{\mathfrak{g}}$ of $\widehat{\mathfrak{g}}_{E, l}^{\ast}$. 
Therefore we obtain 
the morphism (\ref{restrict:Banach}). 
\end{proof}

Let $\mM$ be the moduli stack of coherent sheaves on $X$, 
and we regard it as a complex analytic stack. 
The above lemma implies the following proposition: 
\begin{prop}\label{prop:rest2}
By shrinking $\uU$ if necessary, 
the morphism (\ref{I:analytic})
induces the morphism of complex analytic stacks
\begin{align}\label{mor:stack}
I_{\ast} \colon T \to \mM. 
\end{align}
\end{prop}
\begin{proof}
The map in Lemma~\ref{prop:restrict} corresponds to 
the element 
\begin{align*}
\alpha \in \mathfrak{g}_E^{1} \otimes \Gamma(\oO_T)
\end{align*}
satisfying the MC equation of the 
dg-algebra 
$\mathfrak{g}_E^{\ast} \otimes \Gamma(\oO_T)$. 
Then we obtain 
the dg-$\aA^{0, \ast}(\oO_X) \otimes \oO_T$-module 
\begin{align}\label{A:coh}
(\aA^{0, \ast}(\eE^{\bullet}) \otimes \oO_T, d_{\aA^{0, \ast}(\eE^{\bullet}) \otimes \oO_T}+\alpha). 
\end{align}
Here $\eE^{\bullet}$ is the complex (\ref{seq:E}). 

The dg-module (\ref{A:coh}) 
is a bounded complex of $\oO_{X \times T}$-modules. 
We can show that each cohomology of (\ref{A:coh}) is a 
coherent $\oO_{X \times T}$-module 
as in~\cite[Lemma~4.1.5]{MR2648899},
which essentially follows the argument in~\cite[p51-52]{MR1079726}. 
Indeed for each $t \in T$
and $x \in X$, by the proof of~\cite[Lemma~4.1.5]{MR2648899}
there is an open neighborhood 
$x \in U$ such that 
there is a degree zero $\cC^{\infty}$-isomorphism
\begin{align}\label{phi_t}
\phi_t \colon \aA^{0, \ast}(\eE^{\bullet})|_{U}
\stackrel{\cong}{\to}
\aA^{0, \ast}(\eE^{\bullet})|_{U}
\end{align}
satisfying that 
\begin{align}\notag
\phi_t^{-1} \circ (d_{\aA^{0, \ast}(\eE^{\bullet})}+\alpha_t) \circ
\phi_t =d_{\aA^{0, \ast}(\eE^{\bullet})}+\beta_t.
\end{align}
Here in the notation (\ref{write:alpha}),  
$\beta_t$ is of the form 
\begin{align*}
\beta_t=((\beta_0^{i})_t, 0, 0, \ldots), \ 
(\beta_0^i)_t \in \Hom(\eE^i|_{U}, \eE^{i+1}|_{U}).
\end{align*}
This implies that 
the dg-module (\ref{A:coh})
restricted to $U \times \{t\}$
is gauge equivalent to a complex which is quasi-isomorphic 
to a bounded complex of holomorphic vector bundles on $U$. 
The isomorphism (\ref{phi_t}) can be found 
by solving a certain differential equation, 
as in~\cite[p51-52]{MR1079726}.
As remarked in~\cite[p52]{MR1079726}, 
the solution $\phi_t$ is analytic in $t \in T$
as $\alpha_t$ is. Therefore by shrinking $U$, $T$
if necessary we see that 
(\ref{A:coh}) restricted to $U \times T$
is gauge equivalent to a complex 
which is quasi-isomorphic to a bounded complex of 
analytic vector bundles on $U \times T$. 
In particular, each cohomology of (\ref{A:coh}) is coherent. 

Therefore (\ref{A:coh}) determines an object 
\begin{align*}
\eE_T^{\bullet} \in D^b_{\Coh(X \times T)}(\Modu \oO_{X \times T}).
\end{align*}
We show that by shrinking $\uU$ if necessary, 
the object $\eE_T^{\bullet}$ is 
quasi-isomorphic 
to a $T$-flat 
sheaf 
\begin{align*}
E_T \cneq \hH^0(\eE_T^{\bullet}) \in \Coh(X \times T). 
\end{align*}
By the construction of $\eE_T^{\bullet}$, at $t=0$
we have $\eE_T^{\bullet} \dotimes_{\oO_T} \oO_{\{0\}} \cong E$. 
We have 
the spectral sequence
\begin{align*}
E_2^{p, q}=
\tT or_{-p}^{\oO_{X \times T}}(\hH^q(\eE_T^{\bullet}), \oO_{X \times \{0\}}) \Rightarrow 
\hH^{p+q}(E).
\end{align*}
Let $q_0$ be the maximal 
$q \in \mathbb{Z}$ such that 
$\hH^q(\eE_T^{\bullet}) \neq 0$. 
If $q_0>0$, then 
by the above spectral sequence 
we have $\hH^{q_0}(\eE_T^{\bullet})|_{t=0}=0$. 
Therefore by shrinking $\uU$, we have $q_0 \le 0$, 
and as $E\neq 0$ it follows that $q_0=0$ by the above spectral 
sequence. 
Moreover we have $E_2^{-1, 0}=0$, which implies that 
$E_T$ is flat at $t=0$, hence $E_2^{p, 0}=0$ for 
any $p<0$. 
Then by the above spectral sequence again, 
we have $E_2^{0, -1}=0$, hence we
may assume $\hH^{-1}(\eE_{T}^{\bullet})=0$. 
Inductively, by shrinking $\uU$ we see that $\hH^q(\eE_T^{\bullet})=0$
for any $q<0$. 
Therefore the above claim holds. 

By the universal property of $\mM$, the sheaf $E_T$ defines the 
morphism (\ref{mor:stack}). 
\end{proof}
\begin{prop}\label{prop:rest3}
The morphism of complex analytic stacks 
$I_{\ast} \colon T \to \mM$ in 
 (\ref{mor:stack})
is smooth of relative dimension $\dim \Aut(E)$. 
\end{prop}
\begin{proof}
We first show that $I_{\ast} \colon T \to \mM$
 is smooth. 
Let $(S, s)$ be a complex analytic space and 
$(S, s) \to (\mM, [E])$ a morphism of complex analytic stacks
which sends $s$ to $[E]$. 
It is enough to show
that, 
after replacing $S$ by its open neighborhood at $s \in S$
if necessary, we have the factorization
\begin{align}\label{fact:S}
(S, s) \to (T, 0) \stackrel{I_{\ast}}{\to} (\mM, [E]).
\end{align}
By shrinking $S$ if necessary, we may assume that
$S \to \mM$ factors through 
\begin{align*}
(S, s) \stackrel{f_1}{\to} (A, p)
\to (\mM, [E])
\end{align*}
where the right morphism 
is the local 
atlas in (\ref{atlas}). 
Let $\eE_A^{\bullet}$ be the complex on $X \times A$
constructed in (\ref{seq:Eu}). 
By pulling $\eE_A^{\bullet}$
back by $f^{\ast}_1$, 
we obtain the complex
\begin{align*}
\eE_S^{\bullet}=f^{\ast}_1\eE_A^{\bullet}. 
\end{align*}
Then
as described in Section~\ref{subsec:dg}, 
the complex structures of 
each term of $\eE_S^{\bullet}$ and their differentials 
give rise to the solution of the MC equation of 
the dg-algebra $\mathfrak{g}_E^{\ast} \otimes \oO_S(S)$. 
Thus we obtain a map of Banach analytic spaces
\begin{align*}
f_2 \colon
(S, s) \to (\mathrm{MC}(\mathfrak{g}_E^{\ast}), 0). 
\end{align*}
We are left to prove the existence of the
morphism $f_3 \colon (S, s) \to (T, 0)$ such that 
the composition
\begin{align*}
(S, s) \stackrel{f_3}{\to} (T, 0) \stackrel{I_{\ast}}{\to}  (\mathrm{MC}(\mathfrak{g}_E^{\ast}), 0)
\end{align*}
differs from $f_2$ only up to gauge equivalence. 
The existence of such $f_3$ is proved in~\cite[Theorem~2.2.2]{JuTu}
when $E$ is a vector bundle, and the same 
argument applies for the complex of vector bundles (\ref{seq:E}). 
Below we give an outline of the proof. 

For $y \in \mathfrak{g}_E^{1}$, consider the series
\begin{align*}
P_{\ast}(y) \cneq \sum_{n\ge 1}P_n(y, \ldots, y)
\end{align*}
where $P$ is the homotopy inverse of $I$ in (\ref{P:inverse}). 
By Lemma~\ref{lem:mbound}, there is an open neighborhood 
$0 \in \uU' \subset \mathfrak{g}_E^1$ 
in $\lVert - \rVert_{l}$-norm
such that $P_{\ast}$ 
gives the analytic map
\begin{align*}
P_{\ast} \colon \uU' \to \Ext^1(E, E). 
\end{align*}
Since $P$ is an $A_{\infty}$-homomorphism, 
after shrinking $\uU'$ if necessary 
the above map induces the morphism of Banach analytic spaces
\begin{align*}
P_{\ast} \colon 
\mathrm{MC}(\mathfrak{g}_E^{\ast}) \cap \uU' \to T. 
\end{align*}
Therefore by shrinking $S$ if necessary so that $f_2(S) \subset \uU'$, 
we have the analytic map
\begin{align*}
f_3=
P_{\ast} \circ f_2 \colon (S, s) \to (T, 0). 
\end{align*}
It remains to show that two maps
\begin{align*}
I_{\ast} \circ f_3=I_{\ast} \circ P_{\ast} \circ f_2, \ f_2
\colon (S, s) \to (\mathrm{MC}(\mathfrak{g}_E^{\ast}), 0)
\end{align*}
 are 
gauge equivalent. 
Since $P$ is a homotopy inverse of $I$, there is an 
$A_{\infty}$-homomorphism 
\begin{align*}
H \colon \mathfrak{g}_E^{\ast}  \to \mathfrak{g}_E^{\ast} \otimes \Omega_{[0, 1]}^{\ast}
\end{align*}
such that $H(0)=\id$ and $H(1)=I \circ P$. 
Then $H$ also satisfies the boundedness property as in Lemma~\ref{lem:mbound}
(see~\cite[Corollary~A.2.7]{JuTu}), so that after shrinking $\uU'$ if necessary 
the $A_{\infty}$-homomorphism $H$ induces the analytic map
\begin{align*}
H_{\ast} \colon \mathrm{MC}(\mathfrak{g}_E^{\ast}) \cap \uU'  \to \mathrm{MC}(\mathfrak{g}_E^{\ast} \otimes \Omega_{[0, 1]}^{\ast}). 
\end{align*}
Then the analytic map
\begin{align*}
H_{\ast} \circ f_2 \colon S \to \mathrm{MC}(\mathfrak{g}_E^{\ast} \otimes \Omega_{[0, 1]}^{\ast})
\end{align*}
satisfies 
\begin{align*}
H_{\ast} \circ f_2(0)=f_2, \quad H_{\ast} \circ f_2(1)= I_{\ast} \circ P_{\ast} \circ f_2.
\end{align*}
This implies that $f_2$ and $I_{\ast} \circ P_{\ast} \circ f_2$ are 
gauge equivalent in the sense of~\cite[Definition~2.2.2]{MR1950958}. 
As proved in~\cite[Lemma~2.2.2]{MR1950958}, this notion of gauge equivalence 
coincides with the gauge equivalence in (\ref{isom:gauge}).
Therefore the smoothness of $I_{\ast}$ follows.

Finally, the relative dimension of 
$I_{\ast} \colon T \to \mM$ 
is $\dim \Aut(E)$ since 
the dimension of the tangent space of $T$ at $0$ 
is $\dim \Ext^1(E, E)$, and that of $\mM$ at $[E]$ is 
$\dim \Ext^1(E, E)-\dim \Aut(E)$. 
 \end{proof}

\section{Local descriptions of moduli stacks of semistable sheaves}
\label{sec:thm}
In this section, we use the results in the previous 
sections to prove Theorem~\ref{thm:precise}. 
By applying the arguments to the CY 3-fold case, 
we also obtain Corollary~\ref{cor:CY3}. 

\subsection{Convergent relation of the Ext-quiver}
For a smooth projective variety $X$, 
let
\begin{align*}
E_{\bullet}=(E_1, \ldots, E_k)
\end{align*}
be a simple collection of 
coherent sheaves on $X$, and 
$Q_{E_{\bullet}}$ the associated Ext-quiver
(see Subsection~\ref{subsec:Extquiver}).  
Here we construct a convergent relation of
$Q_{E_{\bullet}}$ from the minimal $A_{\infty}$-structure
on the derived category of coherent sheaves on $X$.

Let us consider the sheaf on $X$ of the form
\begin{align}\label{object:E}
E=\bigoplus_{i=1}^k V_i \otimes E_i
\end{align}
for vector spaces $V_i$, 
and set 
$m_i=\dim V_i$. 
Note that we have the decomposition
\begin{align}\label{Ext:decom0}
\Ext^{\ast}(E, E) &=\bigoplus_{1\le a, b\le k}
\Hom(V_a, V_b) \otimes \Ext^{\ast}(E_a, E_b).
\end{align}
Let us take 
a resolution $\eE^{\bullet} \to E$ as in (\ref{seq:E}). 
From its construction, it naturally decomposes into 
the direct sum of resolutions of $E_i$. 
Namely, let
\begin{align}\notag
0 \to \eE^{-N}_i \stackrel{d^{-N}_i}{\to}
  \cdots \to \eE^{-1}_i \stackrel{d^{-1}_i}{\to} \eE^0_i \to E_i \to 0
\end{align}
be the resolution (\ref{seq:E}) applied for $E_i$. 
By taking $N\gg0$, we may assume that $N$ is independent of $i$. 
Then the complex $\eE^{\bullet}$ in (\ref{seq:E}) is 
\begin{align*}
\eE^{\bullet}=\bigoplus_{i=1}^k V_i \otimes \eE_i^{\bullet}. 
\end{align*}
Therefore we have the decompositions
\begin{align}\label{Ext:decom}
\mathfrak{g}_E^{\ast} &=\bigoplus_{1\le a, b \le k}
\Hom(V_a, V_b) \otimes A^{0, \ast}
(\hH om^{\ast}(\eE_a^{\bullet}, \eE_b^{\bullet})). 
\end{align}
Here $\mathfrak{g}_E^{\ast}$ is the dg-algebra (\ref{minimal:g}), 
defined via the above complex $\eE^{\bullet}$. 
The decomposition of $\mathfrak{g}_E^{\ast}$ is compatible with the 
Laplacian operator $\Delta$. 
Indeed each complex 
$\aA^{0, \ast}(\hH om^{\ast}
(\eE_a^{\bullet}, \eE_b^{\bullet}))$ is elliptic 
and hence we have linear operators
\begin{align*}
&i_{a, b} \colon 
\Ext^{\ast}(E_a, E_b) \hookrightarrow
 A^{0, \ast}(\hH om^{\ast}(\eE_a^{\bullet}, \eE_b^{\bullet})) \\ 
&p_{a, b} \colon 
 A^{0, \ast}(\hH om^{\ast}(\eE_a^{\bullet}, \eE_b^{\bullet}))
\twoheadrightarrow \Ext^{\ast}(E_a, E_b) \\
& h_{a, b} \colon 
 A^{0, \ast}(\hH om^{\ast}(\eE_a^{\bullet}, \eE_b^{\bullet}))
\to  A^{0, \ast-1}(\hH om^{\ast}(\eE_a^{\bullet}, \eE_b^{\bullet}))
\end{align*}
satisfying the same relations as (\ref{relation})
and 
\begin{align}\label{u:sum}
\star=\bigoplus_{1\le a, b\le k} \id_{\Hom(V_a, V_b)} \otimes \star_{a, b}
\end{align}
where $\star$ is either $i$ or $p$ or $h$
given in Subsection~\ref{subsec:minimal}. 

Let 
$\overline{E}$ be the coherent sheaf on $X$
defined by 
\begin{align}\label{E:bar}
\overline{E} \cneq \bigoplus_{i=1}^k E_i
\end{align}
and consider the $A_{\infty}$-product
\begin{align}\label{mn:Eb}
m_n \colon 
\Ext^1(\overline{E}, \overline{E})^{\otimes n} \to \Ext^2(\overline{E}, \overline{E}). 
\end{align}
By the relation (\ref{u:sum}) and 
the explicit formula (\ref{def:mn}) of 
the $A_{\infty}$-product, 
the 
map (\ref{mn:Eb})
only consists of the direct sum factors of the form
\begin{align}\notag
m_n \colon 
\Ext^1(E_{\psi(1)}, E_{\psi(2)}) \otimes 
&\Ext^1(E_{\psi(2)}, E_{\psi(3)}) \otimes
\cdots \\
\cdots \otimes 
&\Ext^1(E_{\psi(n)}, E_{\psi(n+1)})  
\label{factor}
\to \Ext^2(E_{\psi(1)}, E_{\psi(n+1)})
\end{align}
for maps $\psi \colon \{1, \ldots, n+1\} \to \{1, \ldots, k\}$, 
which give a minimal $A_{\infty}$-category structure 
on the dg-category generated by $(E_1, \ldots, E_k)$. 
By taking the dual and the products of (\ref{factor})
for all $n\ge 2$, we obtain 
the linear map
\begin{align*}
\mathbf{m}^{\vee} \cneq \prod_{n\ge 2} m_n^{\vee} \colon 
\Ext^2(\overline{E}, \overline{E})^{\vee} \to 
\prod_{n\ge 2} \bigoplus_{\begin{subarray}{c}
\{1, \ldots, n+1 \} \\ \stackrel{\psi}{\to}  
\{1, \ldots, k\}
\end{subarray}}
&\Ext^1(E_{\psi(1)}, E_{\psi(2)})^{\vee} \otimes \cdots \\
&\cdots \otimes 
\Ext^1(E_{\psi(n)}, E_{\psi(n+1)})^{\vee}. 
\end{align*}
Note that an element of the RHS is 
an element of $\mathbb{C}\lkakko Q_{E_{\bullet}} \rkakko$
by (\ref{f:element}). 
Let $\{\mathbf{o}_1, \ldots, \mathbf{o}_l\}$ be a basis of 
$\Ext^2(\overline{E}, \overline{E})^{\vee}$
and set
\begin{align*}
f_i=\mathbf{m}^{\vee}(\mathbf{o}_i) \in \mathbb{C}\lkakko Q_{E_{\bullet}} \rkakko. 
\end{align*}
Then by Lemma~\ref{lem:mbound}, we 
have $f_i \in \mathbb{C}\{Q_{E_{\bullet}}\}$. 
We obtain the convergent relation of $Q_{E_{\bullet}}$
\begin{align}\label{relation:I}
I_{E_{\bullet}} \cneq (f_1, \ldots, f_l). 
\end{align}

\subsection{Deformations of direct sums of simple collections}
We consider the deformations of sheaves of the form (\ref{object:E}).
By the decomposition (\ref{Ext:decom0}), the 
space $\Ext^1(E, E)$ is identified with the space of 
$Q_{E_{\bullet}}$-representations
\begin{align}\label{identify:Ext}
\Ext^1(E, E)=\mathrm{Rep}_{Q_{E_{\bullet}}}(\vec{m}). 
\end{align} 
Here $\vec{m}$ is the dimension vector of
$Q_{E_{\bullet}}$ given by $m_i=\dim V_i$. 
We also have 
\begin{align}\label{G:aut}
G=\Aut(E)=\prod_{i=1}^k \GL(V_i)
\end{align}
and the adjoint action of $\Aut(E)$ 
on $\Ext^1(E, E)$ coincides with the action 
(\ref{G:act}) under the identification (\ref{identify:Ext}).  
Recall that in (\ref{kappa:U}) and (\ref{I:analytic}), we 
constructed analytic maps 
\begin{align}\label{construct:k}
\kappa \colon \uU \to \Ext^2(E, E), \ 
I_{\ast} \colon \uU \to \widehat{\mathfrak{g}}_{E, l}^{\ast}
\end{align}
for a sufficiently small analytic 
open subset $0 \in \uU \subset \Ext^1(E, E)$. 
Explicitly under the identification (\ref{identify:Ext}), 
for a $Q_{E_{\bullet}}$-representation
\begin{align*}
u=(u_e)_{e \in E(Q_{E_{\bullet}})} \in \uU, \ 
u_e \colon V_{s(e)} \to V_{t(e)},
\end{align*}
we have the following 
identities by the decompositions (\ref{Ext:decom0}), 
(\ref{Ext:decom}), (\ref{u:sum}). 
\begin{align}\label{kI:explict}
&\kappa(u)=\sum_{\begin{subarray}{c}
n\ge 2, \\
\{1, \ldots, n+1\} \stackrel{\psi}{\to}
 \{1, \ldots, k\}
\end{subarray}}
\sum_{e_i \in E_{\psi(i), \psi(i+1)}} 
m_n(e_1^{\vee}, \ldots, e_n^{\vee}) \cdot 
u_{e_n} \circ \cdots \circ u_{e_2} \circ u_{e_1}, \\
&\notag 
I_{\ast}(u)=\sum_{\begin{subarray}{c}
n\ge 2, \\
\{1, \ldots, n+1\} \stackrel{\psi}{\to}
 \{1, \ldots, k\}
\end{subarray}}
\sum_{e_i \in E_{\psi(i), \psi(i+1)}} 
I_n(e_1^{\vee}, \ldots, e_n^{\vee}) \cdot 
u_{e_n} \circ \cdots \circ u_{e_2} \circ u_{e_1}. 
\end{align}
Here for $e \in E_{i, j}$, the element 
$e^{\vee} \in \Ext^1(E_i, E_j)$ 
is defined as in (\ref{e:dual}). 
\begin{lem}\label{lem:Extsat}
There is a saturated 
open subset $\vV$ in $\Ext^1(E, E)$ 
w.r.t. the $G$-action on $\Ext^1(E, E)$, 
satisfying
\begin{align*}
0 \in \vV \subset G \cdot \uU \subset
\Ext^1(E, E)
\end{align*}
such that the maps in 
(\ref{construct:k}) induce $G$-equivariant
analytic maps
\begin{align*}
\kappa \colon \vV \to \Ext^2(E, E), \ 
I_{\ast} \colon \vV \to \widehat{\mathfrak{g}}_{E, l}^{\ast}
\end{align*} 
Here 
$G$ acts on $\Ext^2(E, E)$ and 
$\widehat{\mathfrak{g}}_{E, l}^{\ast}$ by adjoint. 
\end{lem}
\begin{proof}
The formal series $\kappa$ and $I_{\ast}$ in (\ref{kI:explict}) 
are obviously 
$G$-equivariant.  
Therefore for a choice of $\uU$ in 
(\ref{kappa:U}), (\ref{I:analytic}), the maps 
$\kappa$, $I_{\ast}$ can be extended to 
analytic maps
\begin{align*}
\kappa \colon G \cdot \uU \to \Ext^2(E, E), \ 
I_{\ast} \colon G \cdot \uU \to \widehat{\mathfrak{g}}_{E, l}^{\ast}.
\end{align*}
By
Lemma~\ref{lem:saturated2}, there is a saturated analytic open subset 
$\vV \subset G \cdot \uU$ which contains $0 \in \Ext^1(E, E)$, so the 
lemma follows. 
\end{proof}

Let $\vV \subset \Ext^1(E, E)$
be as in Lemma~\ref{lem:Extsat}. 
By Lemma~\ref{lem:saturated}, it is written as 
\begin{align*}
\vV=\pi_{Q_{E_{\bullet}}}^{-1}(V)
\end{align*}
for some analytic 
open subset $0 \in V \subset M_{Q_{E_{\bullet}}}(\vec{m})$, 
where $\pi_{Q_{E_{\bullet}}}$ is the quotient map
\begin{align*}
\pi_{Q_{E_{\bullet}}} \colon \mathrm{Rep}_{Q_{E_{\bullet}}}(\vec{m})
\to M_{Q_{E_{\bullet}}}(\vec{m}). 
\end{align*}
Let 
$R \subset \vV$ be the closed 
analytic subspace given by
\begin{align*}
R \cneq \kappa^{-1}(0) \subset \vV \subset \Ext^1(E, E).
\end{align*}
By the definition of $I_{E_{\bullet}}$ in (\ref{relation:I}), 
under the identification (\ref{identify:Ext}) we have 
\begin{align*}
R=\mathrm{Rep}_{(Q_{E_{\bullet}}, I_{E_{\bullet}})}(\vec{m})|_{V}.
\end{align*}
Here we have used the notation (\ref{Rep:V}) for the RHS. 
Therefore in the notation of Definition~\ref{defi:cmoduli}, we have
\begin{align*}
\mM_{(Q_{E_{\bullet}}, I_{E_{\bullet}})}(\vec{m})|_{V}=[R/G]. 
\end{align*}
\begin{lem}\label{lem:etale}
By shrinking $V$ if necessary, 
the map $I_{\ast}$ in Lemma~\ref{lem:Extsat}
induces the smooth morphism of relative dimension zero
\begin{align}\label{induce:MQ}
I_{\ast} \colon \mM_{(Q_{E_{\bullet}}, I_{E_{\bullet}})}(\vec{m})|_{V} \to 
\mM. 
\end{align}
Here $\mM$ is the moduli stack of coherent sheaves on $X$. 
\end{lem}
\begin{proof}
By Lemma~\ref{prop:restrict}
and Proposition~\ref{prop:rest3}, 
the map $I_{\ast}$ in 
Lemma~\ref{lem:Extsat}
gives 
the analytic maps
\begin{align*}
I_{\ast} \colon R \cap \uU \to \mathrm{MC}(\mathfrak{g}_E^{\ast}), \ 
I_{\ast} \colon R \cap \uU \to \mM. 
\end{align*}
Then by the $G$-equivalence of $I_{\ast}$
and 
the property $\vV \subset G \cdot \uU$
 in Lemma~\ref{lem:Extsat}, 
the above maps extend to the 
$G$-equivariant analytic maps
\begin{align}\label{induce:ana}
I_{\ast} \colon R \to \mathrm{MC}(\mathfrak{g}_E^{\ast}), \ 
I_{\ast} \colon R \to \mM. 
\end{align}
Here the right map is induced by the left map 
as in the proof of Proposition~\ref{prop:rest3}. 
By 
the 
$G$-equivalence of $I_{\ast}$, 
the right map of (\ref{induce:ana}) 
descends to the quotient by $G$
to induce (\ref{induce:MQ}), which is of relative dimension 
zero by Lemma~\ref{lem:Extsat}. 
\end{proof}

\subsection{Functoriality of $I_{\ast}$}\label{subsec:functI}
In this subsection, 
by the explicit description (\ref{kI:explict}) of
the map $I_{\ast}$ 
in Proposition~\ref{prop:complete}, 
we see that it has some functorial property. 
In particular, it implies that $I_{\ast}$ sends 
subsheaves to subrepresentations of Ext-quivers. 
This fact will not be used in the 
rest of this section, but will be used in 
the proof of 
Theorem~\ref{cor:equiv:I}, which will 
be used in 
Theorem~\ref{thm:onedim} to compare 
stability conditions of sheaves and quiver representations. 

For each $i \in V(Q_{E_{\bullet}})=\{1, 2, \ldots, k\}$,
let $V_i, V_i'$ be vector spaces
with dimensions $m_i$, $m_i'$, 
and set 
\begin{align*}
E=\bigoplus_{i=1}^k V_i \otimes E_i, \ 
E'=\bigoplus_{i=1}^k V_i' \otimes E_i. 
\end{align*}
Let us take 
\begin{align}\label{uu'}
u=(u_e)_{e \in E(Q_{E_{\bullet}})}, \ 
u'=(u_e')_{e \in E(Q_{E_{\bullet}})}
\end{align}
where $u_e, u_e'$ are linear maps 
\begin{align*}
u_e \colon V_{s(e)} \to V_{t(e)}, \ 
u_e' \colon V_{s(e)}' \to V_{t(e)}',
\end{align*}
whose operator norms are sufficiently small
so that they give 
$Q_{E_{\bullet}}$-representations
satisfying the relation $I_{E_{\bullet}}$. 
Let $\phi_i \colon V_i \to V_i'$ be linear maps
for $1\le i\le k$ such that the following diagram 
commutes 
for each $e \in E(Q_{E_{\bullet}})$
\begin{align*}
\xymatrix{
V_{s(e)} \ar[r]^-{u_e} 
\ar[d]_-{\phi_{s(e)}} & V_{t(e)} \ar[d]^-{\phi_{t(e)}} \\
V_{s(e)}' \ar[r]_-{u_e'} & V_{t(e)}'.
}
\end{align*}
Then each term of 
\begin{align}\label{Iu:MC}
I_{\ast}(u) \in \mathrm{MC}(\mathfrak{g}_{E}^{\ast}), \ 
I_{\ast}(u') \in \mathrm{MC}(\mathfrak{g}_{E'}^{\ast})
\end{align}
in (\ref{kI:explict}) satisfy 
\begin{align*}
&I_n(e_1^{\vee}, \ldots, e_n^{\vee}) \cdot \phi_{t(e_n)} \circ
u_{e_n} \circ \cdots \circ
u_{e_1} \\
&=I_n(e_1^{\vee}, \ldots, e_n^{\vee}) \cdot 
u_{e_n}' \circ \cdots \circ
u_{e_1}' \circ \phi_{s(e_1)}.
\end{align*}
This implies that  
the map 
\begin{align*}
\bigoplus_{i=1}^k \phi_i \otimes \id \colon
&\left(\aA^{0, \ast}\left(\bigoplus_{i=1}^k V_i \otimes \eE_i^{\bullet}\right), 
d_{\aA^{0, \ast}(\bigoplus_{i=1}^k V_i \otimes \eE_i^{\bullet})}+I_{\ast}(u) \right) \\
& \to 
\left(\aA^{0, \ast}\left(\bigoplus_{i=1}^k V_i' \otimes \eE_i^{\bullet}\right), 
d_{\aA^{0, \ast}(\bigoplus_{i=1}^k V_i' \otimes \eE_i^{\bullet})}+I_{\ast}(u') \right)
\end{align*}
is a map of dg-$\aA^{0, \ast}(\oO_X)$-modules. 
By taking the cohomology of the above map, we 
obtain the morphism of coherent sheaves
\begin{align}\label{mor:coh}
\hH^0 \left(
\bigoplus_{i=1}^k \phi_i \otimes \id \right) \colon
E_{u} \to E_{u'}.
\end{align}
Here $E_{u}$, $E_{u'}$ are coherent sheaves 
corresponding to 
$u$, $u'$
under the map in
Proposition~\ref{prop:rest2} respectively. 

\begin{rmk}\label{rmk:operator}
In the above argument, we assumed that the operator norms 
of $u, u'$ are enough small so that $I_{\ast}$ is defined. 
We can relax this condition in the following cases. 
First suppose that each $\phi_i$ is injective or surjective. 
Then the operator norm of $u$ is bounded by that of $u'$, so 
if the operator norm of $u'$ is enough small then
so is $u$ 
and $I_{\ast}(u)$ is defined. 
Next if 
$u, u'$ correspond to nilpotent 
$Q_{E_{\bullet}}$-representations, then 
whatever the operator norms of $u, u'$
the infinite sums
$I_{\ast}(u), I_{\ast}(u')$
in (\ref{kI:explict})
are finite sums.
So in the above cases,   
$E_u, E_{u'}$ and the morphism (\ref{mor:coh})
 are well-defined. 
\end{rmk}

\subsection{\'Etale slice}
Below we
use the notation in Subsection~\ref{subsec:moduli}. 
Let 
$\mM_{\omega}(v)$ be the moduli 
stack of 
$\omega$-Gieseker semistable sheaves on $X$
with Chern character $v$,  
$M_{\omega}(v)$ its coarse moduli space.
Let $E$ be a polystable sheaf of the form (\ref{polystable}), 
and take closed points
\begin{align*}
p=[E] \in M_{\omega}(v), \ p'=[E] \in \mM_{\omega}(v).
\end{align*}
For $m\gg 0$, let 
$\mathbf{V}$ be the vector space given by
\begin{align*}
\mathbf{V}=H^0(E(m))=\bigoplus_{i=1}^k V_i \otimes H^0(E_i(m)). 
\end{align*}
Let $q \in \mathrm{Quot}^{\circ}(\mathbf{V}, v)$ 
be a point which is mapped to $p'$
under the quotient morphism
$\mathrm{Quot}^{\circ}(\mathbf{V}, v)
\to \mM_{\omega}(v)$. 
Then we have 
\begin{align*}
\Stab_{\GL(\mathbf{V})}(q)=G \subset \GL(\mathbf{V})
\end{align*}
where $G$ is given as in (\ref{G:aut}). 
By Luna's \'etale slice theorem~\cite{MR0342523}, there is an 
affine locally closed $G$-invariant subscheme 
\begin{align*}
q \in Z \subset \mathrm{Quot}^{\circ}(\mathbf{V}, v)
\end{align*}
such that the natural 
$\GL(\mathbf{V})$-equivariant morphism
\begin{align*}
\GL(\mathbf{V}) \times_G Z \to \mathrm{Quot}^{\circ}(\mathbf{V}, v)
\end{align*}
is \'etale. 
Moreover by taking the quotients by $\GL(\mathbf{V})$, 
we obtain the Cartesian diagram
\begin{align}\label{dia:etale}
\xymatrix{
[Z/G]
\ar[r]
\ar[d]_{p_{Z}}\ar@{}[dr]|\square &
 \mM_{\omega}(v) \ar[d]^{p_{M}} \\
Z \sslash G \ar[r] & M_{\omega}(v)
}
\end{align}
such that each horizontal arrows 
are \'etale. Therefore 
there is a saturated 
analytic open subset $\wW \subset Z$
(w.r.t. the $G$-action on $Z$)
which contains $q$ and the 
Cartesian diagram of complex analytic stacks
\begin{align*}
\xymatrix{
[\wW/G]
\ar[r]
\ar[d]_{p_{\wW}}\ar@{}[dr]|\square &
 \mM_{\omega}(v) \ar[d]^{p_{M}} \\
\wW \sslash G \ar[r] & M_{\omega}(v)
}
\end{align*}
such that each horizontal arrows are analytic 
open immersions. 

On the other hand, let us consider the morphism $I_{\ast}$
in Lemma~\ref{lem:etale} applied for 
the above polystable sheaf 
$p'=[E] \in \mM_{\omega}(v)$. 
By the openness of stability, by shrinking $\uU$ in Lemma~\ref{lem:Extsat} if necessary, 
the map $I_{\ast}$ in Lemma~\ref{lem:etale} 
factors through the open substack 
$\mM_{\omega}(v) \subset \mM$:
\begin{align}\label{Iast:M}
I_{\ast} \colon \mM_{(Q_{E_{\bullet}}, I_{E_{\bullet}})}(\vec{m})|_{V} \to 
\mM_{\omega}(v). 
\end{align}
Now the following proposition completes the proof of 
Theorem~\ref{thm:precise}. 
\begin{prop}\label{prop:complete}
By shrinking $\vV$ in Lemma~\ref{lem:Extsat}
and $\wW$ if necessary
(while keeping the condition to be 
saturated in 
$\Ext^1(E, E)$, $Z$ respectively)
 the map (\ref{Iast:M}) induces the commutative isomorphisms
\begin{align}\label{dia:MW}
\xymatrix{
[R/G]=\mM_{(Q_{E_{\bullet}}, I_{E_{\bullet}})}(\vec{m})|_{V} 
\ar[r]_-{\cong}^-{I_{\ast}} \ar[d]_{p_Q} &
[\wW/G] \ar[d]^{p_{\wW}} \\
R\sslash G=M_{(Q_{E_{\bullet}}, I_{E_{\bullet}})}(\vec{m})|_{V}
\ar[r]_-{\cong} & 
\wW \sslash G. 
}
\end{align}
\end{prop}
\begin{proof}
The map (\ref{Iast:M}) induces the analytic map 
$R\sslash G \to M_{\omega}(v)$. 
So by shrinking $0 \in V \subset M_{Q_{E_{\bullet}}}(\vec{m})$ if 
necessary, we may assume
that the above map factors through 
$R\sslash G \to \wW\sslash G$. 
Then we have the commutative diagram
\begin{align*}
\xymatrix{
[R/G]
\ar[r]^-{I_{\ast}} \ar[d]_{p_Q} &
[\wW/G] \ar[d]^-{p_{\wW}} \\
R\sslash G
\ar[r] & 
\wW \sslash G. 
}
\end{align*}
Let $K\subset G$ be a maximal compact subgroup, and 
take a sufficiently small 
$K$-invariant analytic open subset
$q \in \wW_1 \subset \wW$. 
Then as in the proof of Proposition~\ref{prop:rest3}, 
the composition
\begin{align*}
\wW_1 \to \wW \to [\wW/G] \subset \mM_{\omega}(v)
\end{align*}
admits a lift $\phi \colon \wW_1 \to R$
using the homotopy inverse $P$ of $I$. 
Moreover the proof in \textit{loc.cit.} immediately 
implies that $\phi$ can taken to be $K$-equivariant. 
(Indeed if the map $f_2$ in \textit{loc.cit.} is $K$-equivariant, 
then so is $f_3$ as $P_{\ast}$ is $K$-equivariant.)
So we have the commutative diagram
\begin{align}\label{2commute2}
\xymatrix{
R  \ar[d] & \wW_1 \ar[d] \ar[l]_-{\phi} \\
[R/G] \ar[r]^{I_{\ast}} & [\wW/G].
}
\end{align}
Note that the bottom arrow is a smooth morphism of 
relative dimension zero by Lemma~\ref{lem:etale}. 
Let $0 \in R_1 \subset R$ be a sufficiently small 
$K$-invariant analytic open neighborhood. 
Since both of $R_1$ and $\wW_1$ 
are the bases of versal families of flat deformations of $E$
with tangent space $\Ext^1(E, E)$, 
and $\phi$ is isomorphism at the tangent by the 
diagram (\ref{2commute2}), 
the $K$-equivariant map $\phi$
gives an isomorphism
$\psi \colon \wW_1 \stackrel{\cong}{\to} R_1$
for some suitable choices of $\wW_1$, $R_1$.  
By setting $\psi=\phi^{-1}$, 
we obtain the commutative diagram
\begin{align}\label{2commute}
\xymatrix{
R_1 \ar[r]^-{\psi}_-{\cong} \ar[d] & \wW_1 \ar[d] \\
[R/G] \ar[r]^{I_{\ast}} & [\wW/G].
}
\end{align}
By Lemma~\ref{lem:welldef} below, 
after shrinking $R_1$ if necessary we
 can extend the $K$-equivariant isomorphism 
$\psi \colon R_1 \stackrel{\cong}{\to} \wW_1$ to 
a $G$-equivariant isomorphism between 
$G$-invariant open subsets in 
$R$ and $\wW$
\begin{align}\label{isom:R2}
\widetilde{\psi} \colon R_2 \cneq G \cdot R_1 
\stackrel{\cong}{\to} \wW_2 \cneq 
G \cdot \wW_1
\end{align}
by sending $g \cdot x$ to 
$g \cdot \psi(x)$
for $g \in G$ and $x \in R_1$. 
Then by Lemma~\ref{lem:Zsat} below, 
the isomorphism (\ref{isom:R2}) restricts to the 
isomorphism of saturated open subsets. 
By taking the quotients of $G$-actions, we obtain the 
desired isomorphisms (\ref{dia:MW}). 
\end{proof}

In the proof of the above proposition, we postponed the following two lemmas: 
\begin{lem}\label{lem:welldef}
The map (\ref{isom:R2}) is well-defined and 
an isomorphism.
\end{lem}
\begin{proof}
The lemma is essentially 
proved in the proof of~\cite[Theorem~5.5]{JS}. 
In order to show that (\ref{isom:R2}) is well-defined, 
it is enough to show that 
if $g_1 R_1 \cap g_2 R_1 \neq \emptyset$
for $g_1, g_2 \in G$, then we have 
the identity
$g_1 \psi g_1^{-1}=g_2 \psi g_2^{-1}$ on 
$g_1 R_1 \cap g_2 R_1$. 
By applying $g_2^{-1}$, we may assume
that $g_2=1$. Let $G' \subset G$ be the open 
subset given by
\begin{align*}
G' \cneq \{g\in G : gR_1 \cap R_1 \neq \emptyset\}.
\end{align*}
If we define $G''$ to be
\begin{align*}
G'' \cneq \{g \in G' : g\psi g^{-1}=\psi \mbox{ on } g R_1 \cap R_1\}
\end{align*}
then $G''$ is a closed analytic subset of 
$G'$ which contains $K$. 
Therefore if $(G')^{\circ}$, $(G'')^{\circ}$
are the connected components of 
$G'$, $G''$ which contain $K$, then 
we have $(G')^{\circ}=(G'')^{\circ}$. 
Then we take a sufficiently small $K$-invariant
open subset $0 \in R_1' \subset R_1$
satisfying the following: for any
$x_1, x_2 \in R_1'$ with $G \cdot x_1=G\cdot x_2$, 
the connected component of $(G \cdot x_1) \cap R_1$
containing $x_1$ should contain $x_2$. 
The above choice of $R_1'$ implies that
\begin{align*}
G'''\cneq 
\{g \in G : gR_1' \cap R_1' \neq \emptyset\} \subset
(G')^{\circ}.
\end{align*}
Therefore as $(G')^{\circ}=(G'')^{\circ}$, 
for $g \in G'''$ 
we have $g \psi g^{-1}=\psi$ 
on $gR_1' \cap R_1' \neq \emptyset$. 
By replacing $R_1$ with $R_1'$, we see that (\ref{isom:R2}) is well-defined. 
Applying the above argument for the inverse of 
$\psi \colon R_1 \stackrel{\cong}{\to} \wW_1$, we have 
the inverse of (\ref{isom:R2}), showing 
that (\ref{isom:R2}) is an isomorphism. 
\end{proof}

\begin{lem}\label{lem:Zsat}
There exist saturated open subsets 
$\widetilde{\vV} \subset \Ext^1(E, E)$, $\widetilde{\wW} \subset Z$
satisfying $0 \in R \cap \widetilde{\vV} \subset R_2$, 
$q \in \widetilde{\wW} \subset \wW_2$ such that the isomorphism 
(\ref{isom:R2}) restricts to the isomorphism
\begin{align*}
\widetilde{\psi} \colon R \cap \widetilde{\vV} \stackrel{\cong}{\to}
\widetilde{\wW}.
\end{align*}
\end{lem}
\begin{proof}
Let $\wW_3 \subset Z$ be a saturated open subset in $Z$
satisfying 
$q \in \wW_3 \subset \wW_2$, 
which exists by Lemma~\ref{lem:saturated2}, and 
set $R_3 \cneq \widetilde{\psi}^{-1}(\wW_3) \subset R_2$. 
Then $R_3$ is written as 
$R_3=R \cap \vV'$ for some 
$G$-invariant open subset $0 \in \vV' \subset \vV$. 
Let $\vV'' \subset \Ext^1(E, E)$ be a saturated open subset 
satisfying $0 \in \vV'' \subset \vV'$, 
which again exists by Lemma~\ref{lem:saturated2},
 and 
set $R_4 \cneq R \cap \vV'' \subset R_3$. 
Let $\wW_4 \cneq \widetilde{\psi}(R_4)$. We show that 
$\wW_4$ is a saturated open subset in $Z$. 
Indeed for $x \in \wW_4$, 
the orbit closure $\overline{G \cdot x}$ in $Z$
is contained in $\wW_3$ since $\wW_3$ is saturated. 
Take $y \in \overline{G \cdot x}$
and consider $\widetilde{\psi}^{-1}(y) \in R_3$. 
Then since $\vV''$ is saturated, we have 
$\widetilde{\psi}^{-1}(y) \in R_4$, hence $y \in \wW_4$ as desired. 
Now $\vV''$, $\wW_4$ are saturated in $\Ext^1(E, E)$, $Z$. 
By setting $\widetilde{\vV}=\vV''$, 
$\widetilde{\wW}=\wW_4$, we obtain the lemma. 
\end{proof}

\subsection{Calabi-Yau 3-fold case}
We keep the situation in the previous subsections. 
Suppose furthermore that 
$X$ is a smooth projective CY 3-fold, i.e. 
\begin{align*}
\dim X=3, \ \oO_X(K_X) \cong \oO_X.
\end{align*}
In this case, the $A_{\infty}$-structure 
(\ref{mn:Eb})
is cyclic (see~\cite{MR1876072}), i.e. 
for a map 
\begin{align*}
\psi \colon \{1, \ldots, n+1\} \to \{1, \ldots, k\}, \ 
\psi(1)=\psi(n+1)
\end{align*}
and elements
\begin{align*}
a_i \in \Ext^1(E_{\psi(i)}, E_{\psi(i+1)}), \
1\le i\le  n,
\end{align*}
we have the relation
\begin{align}\label{cyclic}
(m_{n-1}(a_1, \ldots, a_{n-1}), a_{n})=(m_{n-1}(a_2, \ldots, a_{n}), a_1). 
\end{align}
Here $m_n$ is the $A_{\infty}$-product 
(\ref{factor}), 
$(-, -)$ is the Serre duality pairing
\begin{align}\label{Serre}
(-, -) \colon 
\Ext^{j}(E_a, E_b) \times \Ext^{3-j}(E_b, E_a)
\to \Ext^3(E_a, E_a) \stackrel{\int_X \mathrm{tr}}{\to} \mathbb{C}.
\end{align} 
Let $W_{E_{\bullet}} \in \mathbb{C}\lkakko Q_{E_{\bullet}} \rkakko$
be defined by
\begin{align}\notag
W_{E_{\bullet}} \cneq \sum_{n\ge 3}
\sum_{\begin{subarray}{c} 
\{1, \ldots, n+1 \} \stackrel{\psi}{\to} \{1, \ldots, k\} \\
\psi(1)=\psi(n+1)
\end{subarray}}
\sum_{e_i \in E_{\psi(i), \psi(i+1)}}
a_{\psi, e_{\bullet}} \cdot e_1 e_2 \ldots e_n. 
\end{align}
Here the coefficient $a_{\psi, e_{\bullet}}$ is given by 
\begin{align}\label{af}
a_{\psi, e_{\bullet}}=\frac{1}{n}(m_{n-1}(e_1^{\vee}, e_2^{\vee}, \ldots, e_{n-1}^{\vee}), 
e_n^{\vee}). 
\end{align}
Then by Lemma~\ref{lem:mbound}, we have
\begin{align*}
W_{E_{\bullet}} \in \mathbb{C}\{ Q_{E_{\bullet}} \} \subset \mathbb{C}\lkakko Q_{E_{\bullet}} \rkakko .
\end{align*}
Therefore $W_{E_{\bullet}}$ determines 
a convergent super-potential of $Q_{E_{\bullet}}$
(see Definition~\ref{def:conv:pot}). 

Let $\overline{E}$ be the object given by (\ref{E:bar}). 
By the Serre duality, 
$\Ext^2(\overline{E}, \overline{E})^{\vee}$
is identified with  
$\Ext^1(\overline{E}, \overline{E})$. 
Thus 
\begin{align}\label{basis}
\{e^{\vee} : e \in E(Q_{E_{\bullet}})\} \subset
\Ext^1(\overline{E}, \overline{E})
\end{align}
gives a basis of $\Ext^2(\overline{E}, \overline{E})^{\vee}$. 
Using this basis, 
the relation $I_{E_{\bullet}}$
defined in (\ref{relation:I}) satisfies
\begin{align*}
I_{E_{\bullet}}=\{\mathbf{m}^{\vee}(e^{\vee}) : e \in E(Q_{E_{\bullet}})\}
=\partial W_{E_{\bullet}}. 
\end{align*}
Here the first identity is due to the definition of $I_{E_{\bullet}}$
via the basis (\ref{basis}), 
and the second identity follows from the 
construction of $W_{E_{\bullet}}$ and the cyclic 
condition (\ref{cyclic}).
As a corollary of Theorem~\ref{thm:precise}, we obtain the following: 
\begin{cor}\label{cor:CY3}
In the situation of Theorem~\ref{thm:precise}, 
suppose furthermore that $X$ is a smooth projective CY 3-fold. 
Then there is a convergent
super-potential $W_{E_{\bullet}}$ of $Q_{E_{\bullet}}$, 
analytic open 
neighborhoods 
$p \in U \subset M_{\omega}(v)$, 
$0 \in V \subset M_{Q_{E_{\bullet}}}(\vec{m})$
and commutative isomorphisms

\begin{align}\label{dia:comiso2}
\xymatrix{
p_M^{-1}(U) 
\ar[d]^{p_M} & \ar[l]_-{\cong}^-{I_{\ast}} \mM_{(Q_{E_{\bullet}, \partial W_{E_{\bullet}})}}(\vec{m})|_{V} 
\ar@{=}[r] \ar[d]^{p_Q}
& \left[ \{d(\tr W_{E_{\bullet}})=0\}/G  \right]
\ar@<-0.3ex>@{^{(}->}[r] & [\pi_Q^{-1}(V)/G] \ar[d]^-{\tr W_{E_{\bullet}}}
\\
U  & 
\ar[l]_-{\cong} M_{(Q_{E_{\bullet}, \partial W_{E_{\bullet}})}}(\vec{m})|_{V} & & \mathbb{C}.
}
\end{align}
Here the bottom arrow sends $0$ to $p$, 
$\pi_Q \colon \mathrm{Rep}_{Q_{E_{\bullet}}}(\vec{m}) 
\to M_{Q_{E_{\bullet}}}(\vec{m})$ is the quotient morphism, 
and $\tr W_{E_{\bullet}}$ is the $G$-invariant
analytic function on the smooth analytic space 
$\pi_Q^{-1}(V)$ (see Subsection~\ref{subsec:potential}). 
\end{cor}

\section{Non-commutative deformation theory}\label{sec:NC}
Note that the diagram (\ref{dia:comiso})
in Theorem~\ref{thm:precise}
in particular implies the isomorphism 
\begin{align}\label{I:nil}
I_{\ast} \colon p_Q^{-1}(0) \stackrel{\cong}{\to} p_M^{-1}(p). 
\end{align}
In this section, we recall the NC deformation theory 
associated to a simple collection of sheaves, and 
explain its relationship to 
the isomorphism (\ref{I:nil}). 
 
More precisely 
in Theorem~\ref{cor:equiv:I}, using NC deformation theory 
 we show that
the map $I_{\ast}$ gives an equivalence of categories 
between the category of nilpotent 
representations of the Ext-quiver and 
the subcategory of coherent sheaves on $X$ 
generated by the given simple collection. 
The result of Theorem~\ref{cor:equiv:I} immediately 
implies the isomorphism (\ref{I:nil}), so giving 
an interpretation of (\ref{I:nil}) via NC deformation theory. 
The result of Theorem~\ref{cor:equiv:I} will be only used in the proof of 
Lemma~\ref{lem:pstab} in the next section, but 
seems to be an interesting result in it's own right
as it gives intrinsic understanding of the isomorphism (\ref{I:nil}). 

\subsection{NC deformation functors}
Let $X$ be a smooth projective variety, 
and take a simple collection of coherent sheaves on it
\begin{align}\label{simple}
E_{\bullet}=(E_1, E_2, \ldots, E_k). 
\end{align}
The NC deformation theory associated to the simple collection 
(\ref{simple}) is formulated for such a 
collection~\cite{Lau, Erik, Kawnc, BoBo}. 
The following convention is due to Kawamata~\cite{Kawnc}. 

By definition, a $k$-\textit{pointed 
$\mathbb{C}$-algebra} is an associative ring $R$
with $\mathbb{C}$-algebra homomorphisms
\begin{align*}
\mathbb{C}^k \stackrel{p}{\to} R 
\stackrel{q}\to \mathbb{C}^k
\end{align*}
whose composition is the identity. 
Then $R$ decomposes as
\begin{align*}
R=\mathbb{C}^{k} \oplus \mathbf{m}, \ 
\mathbf{m} \cneq \Ker q.
\end{align*}
For $1\le i \le k$, let 
$\mathbf{m}_i$ be the kernel of the composition
\begin{align*}
R \stackrel{q}{\to} \mathbb{C}^k \to \mathbb{C}
\end{align*}
where the second map is the $i$-th projection. 
Note that $\mathbf{m}=\cap_{i=1}^{k} \mathbf{m}_i$. 
We define $\aA rt_k$ to be the 
category of finite dimensional $k$-pointed 
$\mathbb{C}$-algebras $R=\mathbb{C}^k \oplus \mathbf{m}$
such that $\mathbf{m}$ is nilpotent. 

For a simple collection (\ref{simple}), 
we have the NC deformation functor
\begin{align}\label{rDef}
\mathrm{Def}_{E_{\bullet}}^{\rm{nc}} \colon 
\aA rt_k  \to \sS et. 
\end{align}
The above functor is defined by sending 
$R=\mathbb{C}^k \oplus \mathbf{m}$ to the set of 
isomorphism classes of pairs 
\begin{align*}
(\eE, \psi), \ \eE \in \Coh(R \otimes_{\mathbb{C}}\oO_X)
\end{align*}
where $\eE$ is a coherent left $R \otimes_{\mathbb{C}}\oO_X$-module 
which is flat
 over $R$, 
and $\psi$ is an isomorphism 
$R/\mathbf{m} \otimes_R \eE \stackrel{\cong}{\to} \oplus_i E_i$
which induces isomorphisms
\begin{align*}
R/\mathbf{m}_i \otimes_R \eE \stackrel{\cong}{\to} E_i, \ 
1\le i\le k.
\end{align*}

\subsection{Pro-representable hull}
Let $\widehat{\aA rt}_k$ be the category whose objects consist of 
$\mathbb{C}^k$-algebras given by inverse
limits of objects in $\aA rt_k$. 
An object $A \in \widehat{\aA rt}_k$ is called a 
\textit{pro-representable hull}
of the functor $\mathrm{Def}_{E_{\bullet}}^{\rm{nc}}$
if there is a formally smooth 
morphism
\begin{align*}
\Hom_{\widehat{\aA rt}_k}(A, -) \to \mathrm{Def}_{E_{\bullet}}^{\rm{nc}}(-)
\end{align*}
which are isomorphisms in first orders.  
A pro-representable hull is, if it exists, 
unique up to non-canonical isomorphisms
(see~\cite{Schle}). 

A pro-representable hull of the functor
$\mathrm{Def}_{E_{\bullet}}^{\rm{nc}}$
is known to exist by~\cite{Lau, Erik}. 
By~\cite{Kawnc}, it 
is explicitly 
constructed by taking the iterated universal extensions
of sheaves $E_i$, which we review here. 
We first set 
$E_i^{(0)}=E_i
$ for $1\le i\le k$. 
Suppose that $E^{(n)}_i$ is constructed
for some $n\ge 0$ and all $1\le i\le k$. 
Then $E^{(n+1)}_i$ is constructed as 
the universal extension
\begin{align}\label{univ:ext}
0 \to \bigoplus_{j=1}^k \Ext^1(E_i^{(n)}, E_j)^{\vee} \otimes E_j 
\to E_i^{(n+1)} \to E_i^{(n)} \to 0. 
\end{align}
Let us set
\begin{align*}
E^{(n)} \cneq \bigoplus_{i=1}^n E_i^{(n)}, \ 
R^{(n)} \cneq \Hom(E^{(n)}, E^{(n)}). 
\end{align*}
Then $R^{(n)}$ is an object of $\aA rt_k$, and 
$E^{(n)}$ 
is an element of $\mathrm{Def}_{E_{\bullet}}^{\rm{nc}}(R^{(n)})$
by~\cite[Theorem~4.8]{Kawnc}. 
Moreover by~\cite[Lemma~4.3, Corollary~4.6, Theorem~4.8]{Kawnc}, 
there exist 
natural surjections 
$R^{(n+1)} \twoheadrightarrow R^{(n)}$ such that 
the inverse limit
\begin{align*}
R_{E_{\bullet}}^{\rm{nc}}=\lim_{\longleftarrow} R^{(n)}
\in \widehat{\aA rt}_k
\end{align*}
is a pro-representable hull of (\ref{rDef}). 
Moreover the surjection 
$E^{(n+1)} \twoheadrightarrow E^{(n)}$ induces the 
isomorphism
\begin{align}\label{RnE}
R^{(n)} \otimes_{R^{(n+1)}}
E^{(n+1)} \stackrel{\cong}{\to} E^{(n)}.
\end{align}
By the surjection 
$R^{(n+1)} \twoheadrightarrow R^{(n)}$, we have the 
fully-faithful embedding 
\begin{align}\label{R:emb}
\modu R^{(n)} \hookrightarrow \modu R^{(n+1)}.
\end{align}
Then the category 
$\modu_{\rm{nil}} R_{E_{\bullet}}^{\rm{nc}}$
is defined by
\begin{align}\label{mod:nil}
\modu_{\rm{nil}} R_{E_{\bullet}}^{\rm{nc}} \cneq 
\lim_{\longrightarrow}
\left(\modu R^{(n)} \right). 
\end{align}
The above category is identified with the 
abelian category of nilpotent finite dimensional 
right $R_{E_{\bullet}}^{\rm{nc}}$-modules. 

\subsection{Equivalence of categories via NC deformations}
In what follows, we show that 
the category 
(\ref{mod:nil}) is equivalent 
to the subcategory
of $\Coh(X)$
\begin{align*}
\langle E_1, E_2, \ldots, E_k \rangle \subset \Coh(X)
\end{align*}
given by the extension closure of $E_1, \ldots, E_k$, i.e. 
the smallest extension closed subcategory of $\Coh(X)$
which contains $E_1, \ldots, E_k$. 
\begin{lem}\label{lem:Phi0}
For $T \in \modu R^{(n)}$, 
we have 
\begin{align}\label{Phi0}
\Phi(T) \cneq T \otimes_{R^{(n)}} E^{(n)} \in 
\langle E_1, \ldots, E_k \rangle. 
\end{align}
\end{lem}
\begin{proof}
Since $R^{(n)} \in \aA rt_k$, it decomposes as 
$R^{(n)}=\mathbb{C}^k \oplus \mathbf{m}^{(n)}$. 
We take the following filtration in $\modu R^{(n)}$
\begin{align*}
\cdots \subset T(\mathbf{m}^{(n)})^j
\subset T(\mathbf{m}^{(n)})^{j-1} \subset 
 \cdots \subset 
T\mathbf{m}^{(n)} \subset T. 
\end{align*}
Then the subquotient
\begin{align*}
T^{(j)} \cneq T(\mathbf{m}^{(n)})^j/T(\mathbf{m}^{(n)})^{j+1}
\end{align*}
is a $\mathbb{C}^k$-module, which is zero 
for $j\gg 0$. 
Since $E^{(n)}$ is an NC deformation of 
$E_{\bullet}$ to $R^{(n)}$, 
it follows that 
$T^{(j)} \otimes_{R^{(n)}}E^{(n)}$ is a direct sum 
of objects in $(E_1, \ldots, E_k)$. 
Since $T$ is given by iterated extensions of 
$T^{(j)}$, the lemma follows. 
\end{proof}
The functor
\begin{align*}
\Phi \colon \modu R^{(n)} \to \langle E_1, \ldots, E_k \rangle
\end{align*}
given by Lemma~\ref{lem:Phi0}
commutes with 
the embedding (\ref{R:emb}) by 
the isomorphism (\ref{RnE}). 
Hence we obtain the functor
\begin{align}\label{Phi}
\Phi \colon \modu_{\rm{nil}}R_{E_{\bullet}}^{\rm{nc}}
\to \langle E_1, \ldots, E_k \rangle. 
\end{align}
Below we show that the functor (\ref{Phi}) is an equivalence 
of categories. We prepare some 
lemmas. 
\begin{lem}\label{lem:Ei}
We have $\Hom(E_i^{(n)}, E_j)=\mathbb{C}^{\delta_{ij}}$
and the natural map
\begin{align*}
\Ext^1(E_i^{(n)}, E_j) \to \Ext^1(E_i^{(n+1)}, E_j)
\end{align*}
is a zero map. 
\end{lem}
\begin{proof}
The lemma follows from 
the exact sequence 
\begin{align*}
0 \to \Hom(E_i^{(n)}, E_j) &\to \Hom(E_i^{(n+1)}, E_j) 
\to \Ext^1(E_i^{(n)}, E_j) \\
 &\stackrel{\id}{\to}
\Ext^1(E_i^{(n)}, E_j) \to \Ext^1(E_i^{(n+1)}, E_j)
\end{align*} 
obtained by applying $\Hom(-, E_j)$ to the exact sequence (\ref{univ:ext}). 
\end{proof}

\begin{lem}\label{lem:Ext0}
For any $U \in \langle E_1, \ldots, E_k \rangle$
and $n\ge 0$, 
the natural map 
\begin{align}\label{nat:EU}
\Ext^1(E_i^{(n)}, U) \to \Ext^1(E_i^{(n+l)}, U)
\end{align}
is a zero map for $l\gg 0$. 
\end{lem}
\begin{proof}
If $U=E_j$ for some $1\le j\le k$, 
the lemma follows from Lemma~\ref{lem:Ei}. 
Otherwise 
there is an exact sequence 
\begin{align*}
0 \to U' \to U \to U'' \to 0, \ 
U', U'' \in \langle E_1, \ldots, E_k \rangle \setminus \{0\}.
\end{align*}
Suppose that the lemma holds for 
$U'$ and $U''$. 
For $l'\gg 0$ and $l'' \gg 0$, 
We have the commutative diagram
\begin{align*}
\xymatrix{
\Ext^1(E^{(n)}, U') \ar[r] \ar[d] & \Ext^1(E^{(n)}, U) \ar[r] \ar[d] &\Ext^1(E^{(n)}, U'') \ar[d]^{0} \\
\Ext^1(E^{(n+l'')}, U') \ar[r] \ar[d]^{0} & \Ext^1(E^{(n+l'')}, U) \ar[r] \ar[d] & \Ext^1(E^{(n+l'')}, U'') \ar[d] \\
\Ext^1(E^{(n+l'+l'')}, U') \ar[r]  & \Ext^1(E^{(n+l'+l'')}, U) \ar[r] & \Ext^1(E^{(n+l+l'')}, U'').
}
\end{align*} 
Here the horizontal arrows are exact sequences. 
The map (\ref{nat:EU}) for $l=l'+l''$ 
is the composition of 
middle vertical arrows, which is zero 
by a diagram chasing. 
Therefore the lemma follows by the induction on the number of 
iterated extensions of $U$ by $E_1, \ldots, E_k$. 
\end{proof}

\begin{lem}\label{lem:terminate}
For any $U \in \langle E_1, \ldots, E_k \rangle$, 
the sequence
\begin{align}\label{terminate}
\Hom(E^{(0)}, U) \subset \Hom(E^{(1)}, U) \subset \cdots \subset
\Hom(E^{(n)}, U) \subset \cdots
\end{align}
terminates for $n\gg 0$. 
\end{lem}
\begin{proof}
The lemma can be proved by the induction on the number of iterated
extensions of $U$ by $E_1, \ldots, E_k$. 
If $U=E_i$ for some $i$, then 
the sequence (\ref{terminate}) terminates 
by Lemma~\ref{lem:Ei}. 
Otherwise there is an exact sequence
\begin{align*}
0 \to E_i \to U \to U' \to 0
\end{align*}
for some $1\le i\le k$ and $U' \in 
\langle E_1, \ldots, E_k \rangle$. 
By applying $\Hom(E^{(n)}, -)$, we obtain 
the exact sequence
\begin{align*}
0 \to \Hom(E^{(n)}, E_i) \to \Hom(E^{(n)}, U) \to 
\Hom(E^{(n)}, U'). 
\end{align*}
By Lemma~\ref{lem:Ei}, it follows that 
\begin{align*}
\hom(E^{(n)}, U) \le \hom(E^{(n)}, U')+1.
\end{align*}
By the induction hypothesis, $\hom(E^{(n)}, U')$ is 
bounded above by a number which is independent of $n$. 
Therefore $\hom(E^{(n)}, U)$ is also bounded above. 
\end{proof}
By Lemma~\ref{lem:terminate}, we have 
the functor
\begin{align}\label{Psi}
\Psi \colon \langle E_1, \ldots, E_k \rangle \to 
\modu_{\rm{nil}} R_{E_{\bullet}}^{\rm{nc}}
\end{align}
sending 
$U$ to $\Hom(E^{(n)}, U)$ for $n\gg 0$. 
\begin{lem}\label{lem:Psiexact}
The functor (\ref{Psi}) is exact. 
\end{lem}
\begin{proof}
It is enough to show that (\ref{Psi}) is right exact. 
Let $0 \to U' \to U \to U'' \to 0$ be an exact sequence in 
$\langle E_1, \ldots, E_k \rangle$. 
For $n\gg 0$ and $l\gg 0$, 
we have the commutative diagram
\begin{align*}
\xymatrix{
\Hom(E^{(n)}, U) \ar[r] \ar[d]
^{\cong} & \Hom(E^{(n)}, U'') \ar[r] \ar[d]^{\cong} & \Ext^1(E^{(n)}, U') \ar[d]^{0} \\
\Hom(E^{(n+l)}, U) \ar[r]  & \Hom(E^{(n+l)}, U'') \ar[r] & \Ext^1(E^{(n+l)}, U'). 
}
\end{align*}
Here the isomorphisms of the left and middle 
vertical arrows 
follow from Lemma~\ref{lem:terminate} and the 
right vertical arrow is a zero map by Lemma~\ref{lem:Ext0}. 
Therefore the 
right bottom horizontal arrow is a zero map, 
which shows that $\Hom(E^{(n)}, U) \to \Hom(E^{(n)}, U'')$ is surjective
for $n\gg 0$. Therefore the functor (\ref{Psi}) is exact. 
\end{proof}

We then show the following proposition: 
\begin{prop}\label{prop:equiv}
The functor (\ref{Phi}) is an equivalence of categories. 
\end{prop}
\begin{proof}
The functor (\ref{Psi})
is a right adjoint functor of $\Phi$, so 
there exist canonical natural transformations
\begin{align*}
\id \to \Psi \circ \Phi(-), \ \Phi \circ \Psi(-) \to \id. 
\end{align*}
It is enough to show that both of them are 
isomorphisms of functors. 

As $E^{(n)}$ is flat over $R^{(n)}$, 
the functor $\Phi$ is exact. 
The functor $\Psi$ is also exact 
by Lemma~\ref{lem:Psiexact}, so 
the compositions 
$\Psi \circ \Phi$, $\Phi \circ \Psi$ are also exact. 
Therefore by the induction on the number of 
iterated extensions by simple objects
and the five lemma, 
it is enough to check the isomorphisms
\begin{align*}
S_i \stackrel{\cong}{\to} \Psi \circ \Phi(S_i), \ 
\Phi \circ \Psi(E_i) \stackrel{\cong}{\to}E_i. 
\end{align*}
Here $S_1, \ldots, S_k$ are simple 
$R^{(0)}=\mathbb{C}^k$-modules. 
Since $\Phi(S_i)=E_i$ and $\Psi(E_i)=S_i$, 
the above isomorphisms are obvious. 
\end{proof}

\subsection{Mauer-Cartan formalism of NC deformations}
We can interpret the NC deformation functor (\ref{rDef}) in terms of 
Mauer-Cartan formalism. 
The argument below is also available in~\cite{ESe}. 

For $R \in \aA rt_k$
with the decomposition $R=\mathbb{C}^k \oplus \mathbf{m}$,  
an argument similar to Subsection~\ref{subsec:complex}
shows that 
\begin{align}\notag
\mathrm{Def}_{E_{\bullet}}^{\rm{nc}}(R)
&\cong 
\mathrm{MC}\left(A^{0, \ast}\left( \hH om^{\ast}
\left(\bigoplus_{i=1}^k \eE_{i}^{\bullet}, \bigoplus_{i=1}^k \eE_{i}^{\bullet}
  \right) \underline{\otimes} \mathbf{m} 
 \right)  \right)/\sim \\
\label{MC:NC}&=\mathrm{MC}\left( \bigoplus_{i, j}A^{0, \ast}
(\hH om^{\ast}(\eE_{i}^{\bullet}, \eE_j^{\bullet})) \otimes_{\mathbb{C}} 
\mathbf{m}_{ij}
\right)/\sim.
\end{align}
Here $\sim$ means gauge equivalence, 
$\underline{\otimes}$ is the 
tensor product of $k$-pointed $\mathbb{C}$-algebras 
(see~\cite[Section~1.3]{ESe}), and 
$\mathbf{m}_{ij}=
\mathbf{e}_i \cdot \mathbf{m}\cdot \mathbf{e}_j$
for the idempotents 
$\{\mathbf{e}_1, \ldots, \mathbf{e}_k\}$ of $R$. 
Then using the $A_{\infty}$-operation
$\{I_n\}_{n\ge 1}$
in Subsection~\ref{subsec:minimal},  
we have the map 
\begin{align}
\label{I:MC:map}
I_{\ast} \colon &
\mathrm{MC}\left( \bigoplus_{i, j}\Ext^{\ast}(E_i, E_j) \otimes_{\mathbb{C}}
\mathbf{m}_{ij} \right) \\
\notag &\to
\mathrm{MC}\left( \bigoplus_{i, j}A^{0, \ast}
(\hH om(\eE_{i}^{\bullet}, \eE_j^{\bullet})) \otimes_{\mathbb{C}} 
\mathbf{m}_{ij}
\right)
\end{align}
which is an isomorphism after taking the quotients by gauge equivalence. 
Here the LHS is the solution of the MC equation of the $A_{\infty}$-algebra
\begin{align*}
\bigoplus_{i, j}\Ext^{\ast}(E_i, E_j) \otimes_{\mathbb{C}}
\mathbf{m}_{ij}
\end{align*}
whose $A_{\infty}$-product is given by (\ref{factor}), 
and the map $I_{\ast}$ is constructed as in (\ref{series:I}).

Let $A$ be the $\mathbb{C}^k$-algebra defined by 
\begin{align}\label{alg:A}
A \cneq \mathbb{C}\lkakko Q_{E_{\bullet}} \rkakko/(f_1, \ldots, f_l)
\end{align}
where $(f_1, \ldots, f_l)$ is the convergent 
relation of $Q_{E_{\bullet}}$ given in 
(\ref{relation:I}). 
We have the tautological identification
\begin{align}\label{tautological}
\mathrm{MC}\left( \bigoplus_{i, j}\Ext^{\ast}(E_i, E_j) \otimes_{\mathbb{C}}
\mathbf{m}_{ij} \right)=\Hom_{\widehat{\aA rt}_k}\left(A, R \right). 
\end{align}
Here 
$(e_{i, j} \otimes r_{i, j})$ in the LHS 
corresponds to 
$A \to R$ given by
\begin{align*}
\Ext^1(E_i, E_j)^{\vee} \supset E_{i, j} \ni z
\mapsto e_{i, j}(z) \cdot r_{i, j}.
\end{align*} 
As proved in~\cite[Proposition~2.13]{ESe}, 
under the above identification
 the gauge equivalence in the LHS corresponds to the 
conjugation by an element in $1+\oplus_{i} \mathbf{m}_{ii}$
in the RHS. 

Thus we see that $A$ is a pro-representable hull
of $\mathrm{Def}_{E_{\bullet}}^{\rm{nc}}$. 
By the uniqueness of pro-representable hull, we have an isomorphism
\begin{align*}
R_{E_{\bullet}}^{\rm{nc}} \cong A
\end{align*}
which commute with maps to $\mathrm{Def}_{E_{\bullet}}^{\rm{nc}}$. 
Combined with Proposition~\ref{prop:equiv}, we have the following 
corollary: 
\begin{cor}\label{cor:equiv}
We have an equivalence of categories
\begin{align}\label{Phi:eq}
\Phi \colon \modu_{\rm{nil}}A \stackrel{\sim}{\to} \langle 
E_1, E_2, \ldots, E_k \rangle. 
\end{align}
Here $A$ is the $\mathbb{C}^k$-algebra (\ref{alg:A}). 
\end{cor}

\subsection{Equivalence of categories via $I_{\ast}$}
Let us take a nilpotent $Q_{E_{\bullet}}$-representation
\begin{align}\label{nilp:u}
u=(u_e)_{e \in E(Q_{E_{\bullet}})}, \ 
u_e \colon V_{s(e)} \to V_{t(e)}.
\end{align}
By the argument in Subsection~\ref{subsec:functI}
and Remark~\ref{rmk:operator}, 
the correspondence $u \mapsto I_{\ast}(u)$ 
forms a functor
\begin{align}\label{funct:I}
I_{\ast} \colon \modu_{\rm{nil}}(A) \to \Coh(X). 
\end{align}
We compare the above functor with the equivalence (\ref{Phi:eq})
in the following proposition: 
\begin{thm}\label{cor:equiv:I}
The functor (\ref{funct:I}) is isomorphic to the 
functor $\Phi$ in (\ref{Phi}).
In particular,  
the functor $I_{\ast}$ in (\ref{funct:I})
is an equivalence of categories
\begin{align*}
I_{\ast} \colon \modu_{\rm{nil}}(A) \stackrel{\sim}{\to}
\langle E_1, E_2, \ldots, E_k \rangle \subset \Coh(X). 
\end{align*}
\end{thm}
 \begin{proof}
Let $A=\mathbb{C}^k \oplus \mathbf{m}$ be 
the decomposition, $\{\mathbf{e}_1, \ldots, \mathbf{e}_k\}$
the idempotents of $A$, and set 
$A^{(n)} \cneq A/\mathbf{m}^{n+1}$, 
$\mathbf{m}^{(n)} \cneq \mathbf{m}/\mathbf{m}^{n+1}$. 
Then for an element $u$ as in (\ref{nilp:u}), 
the compositions of $u_e$ for $e \in E(Q_{E_{\bullet}})$
 along with the 
path in $Q_{E_{\bullet}}$ defines the linear map
\begin{align*}
\mathbf{u} \colon \mathbf{m}_{ij}^{(n)} \to \Hom(V_i, V_j), \ 
\mathbf{m}_{ij}^{(n)}
\cneq 
\mathbf{e}_i \cdot \mathbf{m}^{(n)} \cdot \mathbf{e}_j.
\end{align*}
On the other hand, let  
\begin{align*}
c^{(n)}
 \in 
\mathrm{MC}\left( \bigoplus_{i, j}\Ext^{\ast}(E_i, E_j) \otimes_{\mathbb{C}}
\mathbf{m}_{ij}^{(n)} \right)
\end{align*}
be the canonical element 
corresponding to the surjection $A \twoheadrightarrow A^{(n)}$
under the tautological identity (\ref{tautological}). 
Applying the map (\ref{I:MC:map}),
we obtain 
\begin{align}\label{I:kappan}
I_{\ast} (c^{(n)}) \in \mathrm{MC}\left( \bigoplus_{i, j}A^{0, \ast}
(\hH om^{\ast}(\eE_{i}^{\bullet}, \eE_j^{\bullet})) \otimes_{\mathbb{C}} 
\mathbf{m}_{ij}^{(n)} \right). 
\end{align}
Then for $n\gg 0$, we have the identity
\begin{align}\label{Iu:kappa}
I_{\ast}(u)=\mathbf{u} \circ I_{\ast}(c^{(n)}) 
\in \mathrm{MC}(\mathfrak{g}_E^{\ast}).
\end{align}
Let  
$\fF^{(n)} \in \mathrm{Def}_{E_{\bullet}}^{\rm{nc}}(A^{(n)})$
the NC deformation of $E_{\bullet}$ over $A^{(n)}$
corresponding to 
(\ref{I:kappan}) under the isomorphism (\ref{MC:NC}). 
Note that $\fF^{(n)}$ is 
the universal NC deformation over $A$ pulled 
back by the surjection $A\twoheadrightarrow A^{(n)}$. 
Let $T \in \modu_{\rm{nil}}(A)$ be the object 
given by the $Q_{E_{\bullet}}$-representation $u$.
Then the identity (\ref{Iu:kappa}) implies that 
\begin{align*}
I_{\ast}(T) \cong T \otimes_{A^{(n)}} \fF^{(n)}.
\end{align*}
By the construction of $\Phi$ in (\ref{Phi:eq}), 
which goes back to the construction in Lemma~\ref{lem:Phi0}, 
and the universality of $\fF^{(n)}$, 
we have $\Phi(T)=T \otimes_{A^{(n)}} \fF^{(n)}$.
Therefore the proposition holds. 
\end{proof}
In the diagram (\ref{dia:comiso}), note that 
$p_Q^{-1}(0)$ consists of nilpotent 
$A$-modules and 
$p_M^{-1}(p)$ consists of 
objects in 
the extension closure $\langle E_1, \ldots, E_k \rangle$. 
The above proposition implies that 
the isomorphism (\ref{I:nil}) is induced by the 
universal family over NC deformations.

\section{Moduli spaces of one dimensional semistable sheaves}\label{sec:one}
In this section, we focus on the case of moduli spaces of one dimensional 
semistable sheaves, and prove Theorem~\ref{intro:thm:onedim}. 
\subsection{Twisted semistable sheaves}
Let $X$ be a smooth projective variety, and 
$A(X)_{\mathbb{C}}$ its complexified ample cone 
\begin{align*}
A(X)_{\mathbb{C}} \cneq \{B+i\omega \in \mathrm{NS}(X)_{\mathbb{C}} : 
\omega \mbox{ is ample }\}. 
\end{align*}
Let
\begin{align*}
\Coh_{\le 1}(X) \subset \Coh(X)
\end{align*}
be the abelian subcategory of coherent sheaves whose 
supports have dimensions less than or equal to one. 
For an object $E \in \Coh_{\le 1}(X)$
and $B+i\omega \in A(X)_{\mathbb{C}}$, the 
\textit{$B$-twisted $\omega$-slope} 
$\mu_{B, \omega}(E)$ is defined by
\begin{align*}
\mu_{B, \omega}(E) \cneq 
\frac{\chi(E)-B \cdot \ch_{d-1}(E)}{\omega \cdot \ch_{d-1}(E)} 
\in \mathbb{R} \cup \{\infty\}. 
\end{align*}
Here $d=\dim X$, and we
set $\mu_{B, \omega}(E)=\infty$ if 
$\omega \cdot \ch_{d-1}(E)=0$, i.e. if
$E$ is a zero dimensional sheaf. 
\begin{defi}\label{def:Bw}
An object 
$E \in \Coh_{\le 1}(X)$ is 
$(B, \omega)$-(semi)stable if for any 
non-zero subsheaf $F \subsetneq E$, we have 
the inequality
\begin{align*}
\mu_{B, \omega}(F)<(\le) \mu_{B, \omega}(E). 
\end{align*}
\end{defi}
\begin{rmk}\label{rmk:B=0}
If $B=0$, then $E \in \Coh_{\le 1}(X)$ is 
$(0, \omega)$-(semi)stable iff it is 
$\omega$-Gieseker (semi)stable sheaf. 
\end{rmk}
\begin{rmk}\label{rmk:tensor}
For any integer $k\ge 1$ and a line bundle 
$\lL$ on $X$, we have
\begin{align*}
\mu_{B, \omega}(E)=\mu_{kB, k\omega}(E)=
\mu_{kB+c_1(\lL), k\omega}(E \otimes \lL). 
\end{align*}
In particular if $B, \omega$ are elements of $\mathrm{NS}(X)_{\mathbb{Q}}$
so that $kB, k\omega$ are integral, 
then for a line bundle $\lL$ with 
$c_1(\lL)=-kB$
a sheaf $E \in \Coh_{\le 1}(X)$ is 
$(B, \omega)$-semistable iff 
$E \otimes \lL$ is a $\omega$-Gieseker semistable sheaf. 
\end{rmk}

The $(B, \omega)$-stability condition is interpreted 
in terms of Bridgeland stability conditions~\cite{Brs1} as follows. 
Let $N_1(X) \subset H_2(X, \mathbb{Z})$ be the group of numerical classes of 
algebraic one cycles on $X$
and set 
\begin{align*}
\Gamma_X \cneq N_1(X) \oplus \mathbb{Z}.
\end{align*}
Let $\cl$ be the 
group homomorphism defined by
\begin{align}\label{def:cl}
\cl \colon K(\Coh_{\le 1}(X)) \to \Gamma_X, \ 
E \mapsto ([E], \chi(E))
\end{align}
where $[E]$ is the fundamental one cycle associated to $E$. 
By definition, a \textit{Bridgeland stability condition} on 
$D^b(\Coh_{\le 1}(X))$
w.r.t. the group homomorphism map (\ref{def:cl})
consists of data
\begin{align}\label{def:stab}
\sigma=(Z, \aA), \ Z \colon \Gamma_X \to \mathbb{C}, \ 
\aA \subset D^b(\Coh_{\le 1}(X))
\end{align}
where $Z$ is a group homomorphism, 
$\aA$ is the heart of a bounded t-structure
satisfying some axioms (see~\cite{Brs1, K-S} for details). 
It determines the set of \textit{$\sigma$-(semi)stable objects}: 
$E \in D^b(\Coh_{\le 1}(X))$ is $\sigma$-(semi)stable if 
$E[k] \in \aA$ for some $k\in \mathbb{Z}$, and for any 
non-zero subobject $0\neq F \subsetneq E[k]$ in $\aA$, we have the 
inequality in $(0, \pi]$: 
\begin{align*}
\arg Z(\cl(F))<(\le) \arg Z(\cl(E[k])).
\end{align*}

The set of 
Bridgeland stability conditions (\ref{def:stab})
forms a complex manifold, which 
we denote by $\Stab_{\le 1}(X)$. 
The forgetting map $(Z, \aA) \mapsto Z$ gives a local 
homeomorphism
 \begin{align*}
\Stab_{\le 1}(X) \to (\Gamma_X)_{\mathbb{C}}^{\vee}. 
\end{align*}
For a given element $B+i\omega \in A(X)_{\mathbb{C}}$, let 
$Z_{B, \omega}$ be the group homomorphism 
$\Gamma_X \to \mathbb{C}$ defined by 
\begin{align}\label{ZBw}
Z_{B, \omega}(\beta, m) \cneq -m+(B+i\omega)\beta. 
\end{align}
Then the pair
\begin{align}\label{sigma:Bw}
\sigma_{B, \omega} \cneq (Z_{B, \omega}, \Coh_{\le 1}(X))
\end{align}
determines a point in $\Stab_{\le 1}(X)$. 

It is obvious that 
an object in $\Coh_{\le 1}(X)$ is 
$(B, \omega)$-(semi)stable iff it is 
Bridgeland $\sigma_{B, \omega}$-(semi)stable. 
We also call $(B, \omega)$-(semi)stable sheaves as 
\textit{$\sigma_{B, \omega}$-(semi)stable objects}. 
Moreover the map
\begin{align*}
A(X)_{\mathbb{C}} \to \Stab_{\le 1}(X),  \ 
(B, \omega) \mapsto \sigma_{B, \omega}
\end{align*}
is a continuous injective 
map, whose image is denoted by 
\begin{align*}
U(X) \subset \Stab_{\le 1}(X).
\end{align*}

\subsection{Moduli stacks of twisted semistable sheaves}
For $\sigma=\sigma_{B, \omega} \in U(X)$, 
and $v \in \Gamma_X$, let 
\begin{align*}
\mM_{\sigma}(v) \subset \mM
\end{align*}
be the moduli stack of 
$\sigma$-semistable 
$E \in \Coh_{\le 1}(X)$ with $\cl(E)=v$. 
As in the case of Gieseker
stability, we have the following: 
\begin{lem}\label{stack:twist}
The stack $\mM_{\sigma}(v)$ is an algebraic stack of finite type 
with a projective coarse moduli space $M_{\sigma}(v)$.
So we have the natural morphism
\begin{align*}
p_M \colon \mM_{\sigma}(v) \to M_{\sigma}(v). 
\end{align*}
Moreover for each closed point $p \in M_{\sigma}(v)$, 
the same conclusion of Theorem~\ref{thm:precise}
holds. 
\end{lem}
\begin{proof}
If $B$ and $\omega$ are rational, 
then we can reduce the lemma 
in the case of $B=0$ and $\omega$ is integral
by Remark~\ref{rmk:tensor}.
In that case, 
the lemma follows from Theorem~\ref{thm:precise}. 
In general
by wall-chamber 
structure on the space of Bridgeland 
stability conditions, 
there is a collection of real codimension 
one submanifolds $\{\wW_{j}\}_{j \in J}$ in $A(X)_{\mathbb{C}}$ called 
\textit{walls} such that $\mM_{\sigma}(v)$ is constant if 
$\sigma$ is contained in 
a strata 
\begin{align}\label{strata}
\cap_{j \in J'} \wW_{j} \setminus \cup_{j\notin J' }\wW_j
\end{align}
for some subset $J' \subset J$. 
Each wall is given by 
$\mu_{B, \omega}(\beta, n)=\mu_{B, \omega}(\beta', n')$
for other $(\beta', n') \in \Gamma_X$ which is not 
proportional to $(\beta, n)$, i.e. 
\begin{align*}
(n'\beta-n\beta')\omega =
B\beta' \cdot \omega \beta-B\beta \cdot \omega \beta'. 
\end{align*}
The above equation determines a hypersurface in $A(X)_{\mathbb{C}}$
which contains dense rational points. 
Therefore if $(B, \omega)$ is not rational, then we can 
perturb it in the strata (\ref{strata}) 
 and can assume that $(B, \omega)$ is rational. 
\end{proof}

\subsection{Moduli stacks of semistable Ext-quiver representations}
For
$v \in \Gamma_X$ and
$\sigma=\sigma_{B, \omega} \in U(X)$, 
take a point $p\in M_{\sigma}(v)$. 
Suppose that $p$
is represented by a $(B, \omega)$-polystable 
sheaf $E$ of the form
\begin{align}\label{onedim:poly}
E=\bigoplus_{i=1}^k V_i \otimes E_i
\end{align}
where $E_i \in \Coh_{\le 1}(X)$ is $(B, \omega)$-stable 
with $\mu_{B, \omega}(E_i)=\mu_{B, \omega}(E)$. 
Then we have the Ext-quiver $Q_{E_{\bullet}}$
associated to the simple collection 
\begin{align*}
E_{\bullet}=(E_1, \ldots, E_k),
\end{align*}
together with
a convergent relation $I_{E_{\bullet}}$ as in (\ref{relation:I}).
For $i \in V(Q_{E_{\bullet}})=\{1, 2, \ldots, k\}$, 
let $S_i$ be the one dimensional
$Q_{E_{\bullet}}$-representation corresponding 
to the vertex $i$. 
We denote by 
$K(Q_{E_{\bullet}})$ the 
Grothendieck group of finite dimensional 
$Q_{E_{\bullet}}$-representations, and 
take the group homomorphism
\begin{align*}
\mathbf{dim} \colon K(Q_{E_{\bullet}}) \to 
\Gamma_Q \cneq \bigoplus_{i=1}^k \mathbb{Z} \cdot \mathbf{dim} (S_i)
\end{align*}
by taking the dimension vectors. 

Let us take another stability condition 
\begin{align}\label{sigma+}
\sigma^{+}=\sigma_{B^{+}, \omega^{+}}=
(Z_{B^{+}, \omega^{+}}, \Coh_{\le 1}(X)) \in U(X).
\end{align} 
Then we have the group homomorphism
\begin{align*}
Z_{Q}^{+} \colon K(Q_{E_{\bullet}}) \stackrel{\mathbf{dim}}{\to} \Gamma_Q \to
\mathbb{C}, \ 
[S_i] \mapsto Z_{B^{+}, \omega^{+}}(E_i).
\end{align*}
The above group homomorphism
determines a Bridgeland stability condition on
the category of $Q_{E_{\bullet}}$-representations, and the 
associated (semi)stable representations.
They are described in terms of slope stability 
condition as in Definition~\ref{def:Bw}. 
Let $\mu_Q^{+}$ be the slope 
function on the category of $Q_{E_{\bullet}}$-representations
defined by 
\begin{align*}
\mu_Q^{+}(-) \cneq -\frac{\Ree Z_{Q}^+(-)}{\Imm Z_Q^+(-)}. 
\end{align*}
Note that if $\mathbb{V}$ is a $Q_{E_{\bullet}}$-representation 
with dimension vector 
\begin{align}\label{vec:onedim:m}
\vec{m}=(m_i)_{1\le i\le k}, \ m_i=\dim V_i
\end{align}
then we have the identity
\begin{align}\label{id:slopes}
\mu_Q^{+}(\mathbb{V})=\mu_{B^{+}, \omega^{+}}(E)
\end{align}
where $E$ is given by (\ref{onedim:poly}). 
We have the following definition: 
\begin{defi}
A $Q_{E_{\bullet}}$-representation 
$\mathbb{V}$ is 
$\mu_{Q}^{+}$-(semi)stable if for 
any sub 
$Q_{E_{\bullet}}$-representation $0\neq \mathbb{V}' \subsetneq \mathbb{V}$, 
we have the inequality
\begin{align*}
\mu_Q^{+}(\mathbb{V}') <(\le) \mu_Q^+(\mathbb{V}). 
\end{align*}
\end{defi}
For the dimension vector (\ref{vec:onedim:m}), let
\begin{align*}
\mathrm{Rep}_{Q_{E_{\bullet}}}^{+}(\vec{m}) \subset 
\mathrm{Rep}_{Q_{E_{\bullet}}}(\vec{m})
\end{align*}
be the (Zariski) open subset 
consisting of $\mu_Q^{+}$-semistable 
$Q_{E_{\bullet}}$-representations. 
The above open subset is a GIT semistable locus 
with respect to a certain character of $G$
(see~\cite[Section~3]{MR1315461}). 
The quotients by $G$
\begin{align*}
\mM_{Q_{E_{\bullet}}}^{+}(\vec{m})=
[\mathrm{Rep}_{Q_{E_{\bullet}}}^{+}(\vec{m})/G], \ 
M_{Q_{E_{\bullet}}}^{+}(\vec{m})=
\mathrm{Rep}_{Q_{E_{\bullet}}}^{+}(\vec{m}) \sslash G
\end{align*}
are the moduli stack of $\mu_{Q}^{+}$-semistable 
$Q_{E_{\bullet}}$-representations
with dimension vector $\vec{m}$, and its coarse 
moduli space, respectively. 
We have the commutative diagram
\begin{align*}
\xymatrix{
\mM_{Q_{E_{\bullet}}}^{+}(\vec{m}) 
\ar@<-0.3ex>@{^{(}->}[r] \ar[d]_{p_Q^{+}} & 
\mM_{Q_{E_{\bullet}}}(\vec{m}) \ar[d]^{p_Q} \\
M_{Q_{E_{\bullet}}}^{+}(\vec{m}) \ar[r]_{q_Q} & 
M_{Q_{E_{\bullet}}}(\vec{m}).
}
\end{align*}
Here the vertical arrows are natural 
morphisms to the coarse moduli spaces, 
the top horizontal arrow is an open 
immersion and the bottom horizontal arrow $q_Q$ 
is induced by the universality 
of the GIT quotients. 
Note that $q_Q$ is  
projective due to a general 
argument of affine GIT quotients (see~\cite[Section~6]{MR2004218}). 

Let $0 \in V \subset M_{Q_{E_{\bullet}}}(\vec{m})$ be a 
sufficiently small analytic open subset as in Definition~\ref{defi:cmoduli}. 
Let 
\begin{align*}
\mathrm{Rep}_{(Q_{E_{\bullet}}, I_{E_{\bullet}})}^{+}(\vec{m})|_{V} \subset 
\mathrm{Rep}_{(Q_{E_{\bullet}}, I_{E_{\bullet}})}(\vec{m})|_{V}
\end{align*}
be the open locus
consisting of $\mu_Q^{+}$-semistable representations,
where the RHS is defined as in (\ref{Rep:V}). 
Then we set
\begin{align*}
\mM_{(Q_{E_{\bullet}}, I_{E_{\bullet}})}^{+}(\vec{m})|_{V} & \cneq
[\mathrm{Rep}_{(Q_{E_{\bullet}}, I_{E_{\bullet}})}^{+}(\vec{m})|_{V}/G], \\ 
M_{(Q_{E_{\bullet}}, I_{E_{\bullet}})}^{+}(\vec{m})|_{V} & \cneq
\mathrm{Rep}_{(Q_{E_{\bullet}}, I_{E_{\bullet}})}^{+}(\vec{m})|_{V} \sslash G. 
\end{align*}
Here 
$M_{(Q_{E_{\bullet}}, I_{E_{\bullet}})}^{+}(\vec{m})|_{V}$
is the analytic Hilbert quotient 
given in Lemma~\ref{lem:Zquot2}, which is 
a closed analytic subspace of 
$V^+=q_Q^{-1}(V)$. 
We have the commutative diagram
\begin{align}\label{dia:quiver}
\xymatrix{
\mM_{(Q_{E_{\bullet}}, I_{E_{\bullet}})}^{+}(\vec{m})|_{V} 
\ar@<-0.3ex>@{^{(}->}[r] \ar[d]_{p_{(Q, I)}^{+}} \ar[dr]_-{r_{(Q, I)}}& 
\mM_{(Q_{E_{\bullet}}, I_{E_{\bullet}})}(\vec{m})|_{V} \ar[d]^{p_{(Q, I)}} \\
M_{(Q_{E_{\bullet}}, I_{E_{\bullet}})}^{+}(\vec{m})|_{V} \ar[r]_{q_{(Q, I)}} & 
M_{(Q_{E_{\bullet}}, I_{E_{\bullet}})}(\vec{m})|_{V}.
}
\end{align}
Here the vertical arrows are natural 
morphisms to the coarse moduli spaces, 
the top horizontal arrow is an open 
immersion and the bottom horizontal arrow $q_{(Q, I)}$ 
is induced by the universality of analytic Hilbert quotients
(see Lemma~\ref{lem:universal}). 
\begin{lem}\label{lem:proj}
The morphism $q_{(Q, I)}$ in the diagram (\ref{dia:quiver}) is projective. 
\end{lem}
\begin{proof}
We have the following commutative diagram
\begin{align*}
\xymatrix{
M_{(Q_{E_{\bullet}}, I_{E_{\bullet}})}^{+}(\vec{m})|_{V} \ar[d]_{q_{(Q, I)}}
\ar@<-0.3ex>@{^{(}->}[r] & V^{+} \ar[d] \ar@<-0.3ex>@{^{(}->}[r] 
\ar@{}[dr]|\square
& M_{Q_{E_{\bullet}}}^{+}(\vec{m}) \ar[d]^{q_Q} \\
M_{(Q_{E_{\bullet}}, I_{E_{\bullet}})}(\vec{m})|_{V} 
\ar@<-0.3ex>@{^{(}->}[r] & V  \ar@<-0.3ex>@{^{(}->}[r] 
& M_{Q_{E_{\bullet}}}(\vec{m}).
}
\end{align*}
Here the right diagram is a Cartesian square whose horizontal 
arrows are open immersions, and the horizontal arrows 
in the left diagram are closed immersions. 
Since $q_{Q}$ is projective, the
morphism $q_{(Q, I)}$ is projective by the above diagram. 
\end{proof}

\subsection{Moduli stacks of semistable sheaves under the change of stability}
Let us take $\sigma^{+}$ in (\ref{sigma+})
sufficiently close to $\sigma$. 
Then 
by wall-chamber structure on $U(X)$, 
any $\sigma^{+}$-semistable object $E$ 
with $\cl(E)=v$ is $\sigma$-semistable.   
Then
we have the commutative diagram
\begin{align}\label{dia:M+}
\xymatrix{
\mM_{\sigma^{+}}(v) \ar@<-0.3ex>@{^{(}->}[r] \ar[dr]_{r_M}
 \ar[d]_{p_M^{+}}
 & \mM_{\sigma}(v) 
\ar[d]^{p_M} \\
M_{\sigma^{+}}(v) \ar[r]_{q_M}  & M_{\sigma}(v). 
}
\end{align}
Here 
the vertical arrows are natural morphisms to the coarse moduli 
spaces, 
the top arrow is an open immersion
and the bottom arrow is induced by 
the universality of coarse moduli spaces. 
The following is the main 
result in this section. 
\begin{thm}\label{thm:onedim}
For a closed point $p\in M_{\sigma}(v)$
represented by a polystable sheaf (\ref{onedim:poly}), 
there is an analytic 
open neighborhoods $p \in U \subset M_{\sigma}(v)$
and $0 \in V \subset M_{Q_{E_{\bullet}}}(\vec{m})$, 
where $Q_{E_{\bullet}}$ is the Ext-quiver associated to $p$
with convergent relation $I_{E_{\bullet}}$, 
and the dimension vector $\vec{m}$ is given by (\ref{vec:onedim:m}), 
such that 
the diagram (\ref{dia:M+}) pulled back to $U$
\begin{align}\notag
\xymatrix{
r_M^{-1}(U) \ar@<-0.3ex>@{^{(}->}[r] 
 \ar[d]_{p_M^{+}}
 & p_M^{-1}(U)
\ar[d]^{p_M} \\
q_M^{-1}(U) \ar[r]_{q_M}  & U 
}
\end{align}
is isomorphic to the diagram (\ref{dia:quiver}). 
\end{thm}
\begin{proof}
We take $U=\wW \sslash G$, 
$V \subset M_{Q_{E_{\bullet}}}(\vec{m})$
 and the isomorphism 
\begin{align}\label{isom:Iast}
I_{\ast} \colon \mM_{(Q_{E_{\bullet}}, I_{E_{\bullet}})}(\vec{m})|_{V}
\stackrel{\cong}{\to} p_M^{-1}(U)
\end{align}
as in Proposition~\ref{prop:complete}.
It is enough to show that 
the isomorphism (\ref{isom:Iast})
restricts to the isomorphism
\begin{align}\label{desire}
I_{\ast} \colon \mM_{(Q_{E_{\bullet}}, I_{E_{\bullet}})}^{+}(\vec{m})|_{V}
\stackrel{\cong}{\to} r_M^{-1}(U). 
\end{align}
 
For a $\mathbb{C}$-valued point
$x \in \mM_{(Q_{E_{\bullet}}, I_{E_{\bullet}})}(\vec{m})|_{V}$, 
let $\mathbb{V}_x$ be the corresponding 
$Q_{E_{\bullet}}$-representation, 
and $E_x \in \Coh_{\le 1}(X)$
the $(B, \omega)$-semistable sheaf
corresponding to  
$I_{\ast}(x) \in p_M^{-1}(U)$. 
Let $\zZ \subset \mM_{(Q_{E_{\bullet}}, I_{E_{\bullet}})}^{+}(\vec{m})|_{V}$
be the closed substack given by
\begin{align*}
\zZ \cneq \{ x\in \mM_{(Q_{E_{\bullet}}, I_{E_{\bullet}})}^{+}(\vec{m})|_{V} :
I_{\ast}(x) \notin r_M^{-1}(U)\}.
\end{align*}
Namely 
$x \in \mM_{(Q_{E_{\bullet}}, I_{E_{\bullet}})}(\vec{m})|_{V}$
is a $\mathbb{C}$-valued point of $\zZ$ iff
$\mathbb{V}_x$ is $\mu_Q^{+}$-semistable but 
$E_x$ is not $(B^{+}, \omega^{+})$-semistable. 
Below we use the notation in the diagram (\ref{dia:quiver}). 
By Lemma~\ref{lem:pstab} below, 
we have 
\begin{align}\label{h:emptyset}
\zZ \cap (r_{(Q, I)})^{-1}(0)=\emptyset. 
\end{align}
On the other hand, by
Lemma~\ref{lem:univ:prepare}
the subset 
\begin{align*}
p_{(Q, I)}^{+}(\zZ) \subset
 M_{(Q_{E_{\bullet}}, I_{E_{\bullet}})}^{+}(\vec{m})|_{V}
\end{align*}
is closed. Together with Lemma~\ref{lem:proj}, we see that 
\begin{align*}
r_{(Q, I)}(\zZ)=q_{(Q, I)} \circ p_{(Q, I)}^{+}(\zZ) \subset
 M_{(Q_{E_{\bullet}}, I_{E_{\bullet}})}(\vec{m})|_{V}
\end{align*}
is a closed subset. 
By (\ref{h:emptyset}), the above closed subset does not contain 
$0$. Therefore by shrinking $V$ if necessary, we may assume that 
$\zZ=\emptyset$, i.e. (\ref{isom:Iast})
takes
$\mM^{+}_{(Q_{E_{\bullet}}, I_{E_{\bullet}})}(\vec{m})|_{V}$
to $r_M^{-1}(U)$.

Next 
for $x \in \mM_{(Q_{E_{\bullet}}, I_{E_{\bullet}})}(\vec{m})|_{V}$, 
suppose that 
$E_x$ is $(B^{+}, \omega^{+})$-semistable, i.e. 
$I_{\ast}(x) \in r_M^{-1}(U)$. 
Note that by (\ref{id:slopes}), we have
\begin{align}\notag
\mu_{Q}^{+}(\mathbb{V}_x)=\mu_{B^{+}, \omega^{+}}(E_x).
\end{align}
By the functoriality of $I_{\ast}$
in Subsection~\ref{subsec:functI}
and the above equality, 
if a sub $Q_{E_{\bullet}}$-representation 
$\mathbb{V}' \subset \mathbb{V}_x$ 
destabilizes $\mathbb{V}_x$ in 
$\mu_Q^{+}$-stability, then 
by applying $I_{\ast}$ and noting 
Remark~\ref{rmk:operator} we obtain 
the subsheaf $E' \subset E_x$
 which destabilizes $E_x$ 
in $(B^{+}, \omega^{+})$-stability. 
This is a contradiction, so 
$\mathbb{V}_x$ is $\mu_Q^+$-semistable, i.e. 
$x \in \mM^{+}_{(Q_{E_{\bullet}}, I_{E_{\bullet}})}(\vec{m})|_{V}$. 
Therefore we obtain the desired isomorphism (\ref{desire}).  
\end{proof}

We have used the following lemma: 

\begin{lem}\label{lem:pstab}
Under the equivalence $I_{\ast}$ in 
Theorem~\ref{cor:equiv:I},
an object $\mathbb{V} \in \modu_{\rm{nil}}(A)$
with $\dim \mathbb{V}=\vec{m}$
is $\mu_{Q}^+$-semistable iff
$F=I_{\ast}(\mathbb{V})$
is $(B^{+}, \omega^{+})$-semistable in $\Coh_{\le 1}(X)$. 
\end{lem}
\begin{proof}
The if direction is proved in the 
first part of the proof of Theorem~\ref{thm:onedim}, 
so we only prove the only if direction. 
Suppose by contradiction 
that $\mathbb{V}$ is $\mu_Q^+$-semistable 
but $F$ is not 
$(B^{+}, \omega^{+})$-semistable. 
Then there is a non-zero 
subsheaf $F' \subsetneq F$
such that 
$\mu_{B^{+}, \omega^+}(F')>\mu_{B^{+}, \omega^+}(F)$.
On the other hand, as $\sigma^+$ is sufficiently close to 
$\sigma$ we may assume that there is no wall between $\sigma$
and $\sigma^+$ w.r.t. the numerical class $\cl(F)$.
So we have $\mu_{B, \omega}(F')\ge \mu_{B, \omega}(F)$. 
Since $F \in \langle E_1, \ldots, E_k \rangle$
and each $E_i$ is $(B, \omega)$-stable with the same slope, 
the sheaf $F$ is $(B, \omega)$-semistable. 
Therefore we have $\mu_{B, \omega}(F')\le \mu_{B, \omega}(F)$, 
thus $\mu_{B, \omega}(F')=\mu_{B, \omega}(F)$ and
$F'$ is also $(B, \omega)$-semistable. 
By the uniqueness of JH factors of 
$(B, \omega)$-semistable sheaves, we have 
$F' \in \langle E_1, \ldots, E_k \rangle$. 
Then by the equivalence $I_{\ast}$ in 
Theorem~\ref{cor:equiv:I}, 
we find a subobject
$\mathbb{V}' \subset \mathbb{V}$
in $\modu_{\rm{nil}}(A)$ 
with $I_{\ast}(\mathbb{V}') \cong F'$. 
By the identity (\ref{id:slopes}), the subobject $\mathbb{V}'$
destabilizes $\mathbb{V}$, hence a contradiction. 
\end{proof}

\providecommand{\bysame}{\leavevmode\hbox to3em{\hrulefill}\thinspace}
\providecommand{\MR}{\relax\ifhmode\unskip\space\fi MR }
\providecommand{\MRhref}[2]{%
  \href{http://www.ams.org/mathscinet-getitem?mr=#1}{#2}
}
\providecommand{\href}[2]{#2}

Kavli Institute for the Physics and 
Mathematics of the Universe, University of Tokyo,
5-1-5 Kashiwanoha, Kashiwa, 277-8583, Japan.

\textit{E-mail address}: yukinobu.toda@ipmu.jp

\end{document}